\documentclass[11pt]{amsart}
\usepackage[utf8]{inputenc}
\usepackage[T1]{fontenc}

\usepackage[letterpaper, total={6.5in, 9.1in}]{geometry}
\calclayout
\usepackage[usenames, dvipsnames]{color}
\usepackage{graphicx}
\usepackage{amssymb}
\usepackage{amsmath}
\usepackage{amsthm}
\usepackage{mathtools}
\usepackage{hyperref}
\hypersetup{colorlinks=true,allcolors=black}
\usepackage{hypcap}
\usepackage[mathscr]{eucal}
\usepackage{tikz}
\usepackage{lipsum}
\usepackage{array}
\usepackage{booktabs}
\usepackage{siunitx}
\usepackage{dsfont}
\usepackage{caption}
\usepackage{subcaption}
\usepackage{multirow}
\usetikzlibrary{decorations.pathreplacing}

\mathtoolsset{showonlyrefs}

\newcommand{\C}{\mathbb C}

\newcommand{\R}{\mathbb R}
\newcommand{\N}{\mathbb N}

\newcommand{\Z}{\mathbb Z}
\newcommand{\de}{\, \mathrm{d}}

\newcommand{\norm}[1]{\left\Vert #1 \right\Vert}
\newcommand{\abs}[1]{\left| #1 \right|}

\newcommand{\floor}[1]{\left\lfloor #1 \right\rfloor}
\newcommand{\ceil}[1]{\left\lceil #1 \right\rceil}
\newcommand{\del}{\partial}
\newcommand{\1}[1]{\mathds{1}\left\{#1\right\}}

\DeclareMathOperator{\DN}{DN}
\DeclareMathOperator{\ND}{ND}

\newtheorem{thm}{Theorem}[section]
\newtheorem{conj}[thm]{Conjecture}
\newtheorem{lem}[thm]{Lemma}
\newtheorem{prop}[thm]{Proposition}

\theoremstyle{definition}
\newtheorem{defi}[thm]{Definition}
\newtheorem{rem}[thm]{Remark}

\begin{document}
\title{Asymptotics of sloshing eigenvalues for a triangular prism}
\author{Julien Mayrand, Charles Senécal, Simon St-Amant}

\begin{abstract}
We consider the three-dimensional sloshing problem on a triangular prism whose angles with the sloshing surface are of the form $\frac{\pi}{2q}$, where $q$ is an integer. We are interested in finding a two-term asymptotic expansion of the eigenvalue counting function. When both angles are $\frac{\pi}{4}$, we compute the exact value of the second term. As for the general case, we conjecture an asymptotic expansion by constructing quasimodes for the problem and computing the counting function of the related quasi-eigenvalues. These quasimodes come from solutions of the sloping beach problem and correspond to two kinds of waves, edge waves and surface waves. We show that the quasi-eigenvalues are exponentially close to real eigenvalues of the sloshing problem. The asymptotic expansion of their counting function is closely related to a lattice counting problem inside a perturbed ellipse where the perturbation is in a sense random. The contribution of the angles can then be detected through that perturbation.
\end{abstract}

\maketitle

\section{Introduction}

\subsection{The Steklov and sloshing problems}

Let $\Omega \subset \R^n$ be a bounded domain with boundary $\Gamma$ and let $\rho \in L^\infty(\Gamma, \R)$ be a non-negative weight function. The Steklov problem with weight $\rho$ consists of finding all solutions $u \in H^1(\Omega)$ and $\sigma \in \R$ of the problem
\begin{equation}\label{eq:steklov}
\begin{cases}
\Delta u = 0 & \text{in } \Omega, \\
\del_\nu u = \sigma \rho u & \text{on } \Gamma,
\end{cases}
\end{equation}
where $\Delta = \sum_{i=1}^n \del_{x_i}^2$ and $\del_\nu$ denotes the exterior normal derivative on the boundary. The classical Steklov problem consists in having $\rho \equiv 1$ on $\Gamma$.

Our main interest is the sloshing problem. Given a partition of the boundary $\Gamma = \Gamma_N \sqcup \Gamma_S$, the sloshing problem consists of solving \eqref{eq:steklov} with $\rho \equiv 0$ on $\Gamma_N$ and $\rho \equiv 1$ on $\Gamma_S$. It is a mixed Steklov-Neumann boundary problem describing the oscillations of an ideal fluid in a tank shaped like $\Omega$ with walls $\Gamma_N$ and free surface (or sloshing surface) $\Gamma_S$. The admissible values of $\sigma$ are called the sloshing eigenvalues.

\subsection{Our problem} \label{sec:problem}

Let $\Sigma \subset \R^2$ be a triangle with a side $\mathcal{S} = [0, L] \times \{0\}$ of length $L$ making angles $\alpha$ at $(0, 0)$ and $\beta$ at $(L, 0)$ with the other sides. We denote the union of those two other sides by $\mathcal{W}$. Given $M > 0$, we consider the sloshing problem on the rectangular prism $\Omega = \Sigma \times [0, M] \subset \R^3$ with sloshing surface $\Gamma_S = \mathcal{S} \times [0, M]$ and walls 
\begin{equation}
\Gamma_N = (\mathcal{W} \times [0, M]) \cup (\Sigma \times \{0\}) \cup (\Sigma \times\{M\}).
\end{equation}
All this notation is summarized in Figure \ref{fig:domain} where the sloshing surface is shaded in grey.

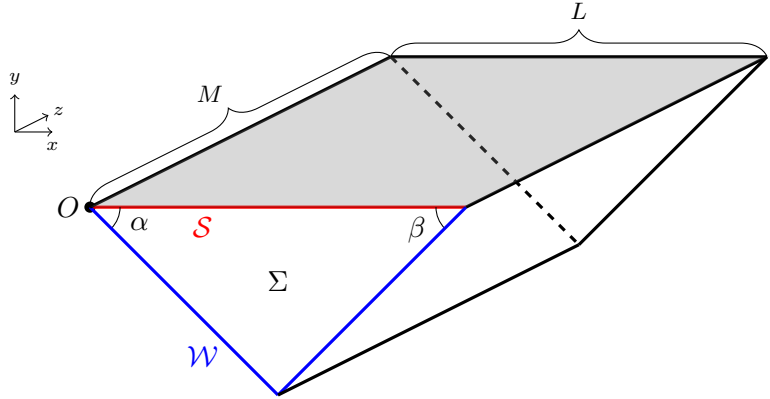
\begin{figure}[!ht]
\begin{center}
\begin{tikzpicture}[scale=0.5]
\coordinate (A) at (0,0);
\coordinate (B) at (10,0);
\coordinate (C) at (18,4);
\coordinate (D) at (8,4);
\coordinate (E) at (5,-5);
\coordinate (F) at (13,-1);
\coordinate (O) at (-2,2);
\coordinate (X) at (-1,2);
\coordinate (Y) at (-2,3);
\coordinate (Z) at (-1.1056,2.4472);
\coordinate (ZO) at (-1.25,2.5472);

\draw[->] (O) -- (X);
\draw[->] (O) -- (Y);
\draw[->] (O) -- (Z);

\draw (X) node[anchor = north]{\tiny $x$};
\draw (Y) node[anchor = south]{\tiny $y$};
\draw (ZO) node[anchor = west]{\tiny $z$};

\draw (0,0) node[circle, fill, inner sep=1.5pt]{};
\draw (0,0) node[anchor = east]{$O$};

\draw[very thick, red] (A) -- (B);

\draw (5,-2.5) node[anchor = south]{$\Sigma$};
\draw[blue] (3,-4.5) node[anchor = south]{$\mathcal{W}$};
\draw[red] (3,0) node[anchor = north]{$\mathcal{S}$};

\draw (0.8,0) node[anchor=north west]{$\alpha$} arc (0:-45:0.8);
\draw (9.2,0) node[anchor=north east]{$\beta$} arc (180:225:0.8);

\draw[very thick, blue] (A) -- (E);
\draw[very thick, blue] (E) -- (B);

\draw[very thick,black] (D) -- (C);

\draw[very thick,black,dashed] (D) -- (F);
\draw[very thick,black] (F) -- (C);

\draw[very thick, black] (A) -- (D);
\draw[very thick, black] (B) -- (C);
\draw[very thick, black] (E) -- (F);

\draw[fill=gray,opacity=0.3] (A) -- (B) -- (C) -- (D);

\draw [decorate,decoration={brace,amplitude=10pt},xshift=0pt,yshift=0.4pt] (8,4) -- (18,4) node [black,midway,yshift=0.6cm] {\footnotesize $L$};
\draw [decorate,decoration={brace,amplitude=10pt},xshift=-0.2pt,yshift=0.2pt] (0,0) -- (8,4) node [black,midway,xshift=-0.4cm,yshift=0.5cm] {\footnotesize $M$};
\end{tikzpicture}
\end{center}
\caption{Example of domain $\Omega$ with $\alpha = \beta = \frac{\pi}{4}$.}
\label{fig:domain}
\end{figure}

The sloshing problem on $\Omega$ consists of finding functions $\Phi :\Omega \rightarrow \C$ such that
\begin{equation}\label{sloshing}
\begin{cases}
\Delta \Phi = 0 & \text{in } \Omega, \\
\del_\nu \Phi = 0 & \text{on } \Gamma_N, \\
\del_\nu \Phi = \sigma \Phi & \text{on } \Gamma_S.
\end{cases}
\end{equation}
for some $\sigma \in \R$. It is a mixed Steklov-Neumann boundary problem describing the oscillations of an ideal fluid in a tank shaped like $\Omega$. The sloshing eigenvalues correspond to the eigenvalues of the Dirichlet-to-Neumann map $\DN : H^{1/2}(\Gamma_S) \rightarrow H^{-1/2}(\Gamma_S)$ which maps a function $u$ to $\del_\nu \tilde{u}$ where $\tilde{u}$ is the solution to
\begin{equation}
\begin{cases}
\Delta \tilde{u} = 0 & \text{in } \Omega, \\
\del_\nu \tilde{u} = 0 & \text{on } \Gamma_N, \\
\tilde{u} = u & \text{on } \Gamma_S.
\end{cases}
\end{equation}
It is a positive semi-definite self-adjoint operator with compact resolvent. As such, its eigenvalues form a discrete sequence
\begin{equation}
0 = \sigma_0 < \sigma_1 \leq \sigma_2 \leq \dots \nearrow \infty
\end{equation}
accumulating at infinity. By separating variables (see \cite[Lemma 2.1]{cuboids}), it is sufficient to consider functions of the form
\begin{equation}\label{eq:Phi}
\Phi(x,y,z) = \cos(\lambda_n z) \varphi(x,y)
\end{equation}
with $\lambda_n = \frac{n\pi}{M}$ where $\varphi : \Sigma \rightarrow \R$ satisfies
\begin{equation}\label{prob:varphi}
\begin{cases}
\Delta \varphi = \lambda_n^2 \varphi & \text{in } \Sigma, \\
\del_\nu \varphi = 0 & \text{on } \mathcal{W}, \\
\del_\nu \varphi = \sigma \varphi & \text{on } \mathcal{S}.
\end{cases}
\end{equation}

We are interested in the asymptotic expansion of the eigenvalue counting function
\begin{equation}
N(\sigma) := \#\{j \in \N_0 : \sigma_j < \sigma\}.
\end{equation}
From \cite{Ag}, we know that
\begin{equation}
N(\sigma) = \frac{LM}{4\pi} \sigma^2 + o(\sigma^2).
\end{equation}
This asymptotic does not capture the contribution from the angles $\alpha$ and $\beta$. Our goal is to find a suitable second term in the asymptotic expansion for $N(\sigma)$ which reveals how both angles affect the counting function. We will be more particularly interested in the case where $\alpha = \frac{\pi}{2q}$ and $\beta = \frac{\pi}{2r}$ for some integers $q$ and $r$ greater or equal to $1$, but not both $1$.

\begin{rem}\label{rem:pi/2}
The case $\alpha = \beta = \frac{\pi}{2}$ obviously does not result in a triangular prism and would actually give rise to an unbounded domain. However, the asymptotic behavior of the sloshing eigenvalues should only depend on a neighborhood of the sloshing surface. This intuition is supported by the following computation. Consider the cuboid $\Omega = [0,L] \times [0,R] \times [0,M] \subset \R^3$ with the sloshing surface corresponding to $y = R$. As above, we can separate variables to get eigenfunctions of the form $\cos(\lambda_n z) \varphi(x,y)$ with $\varphi$ satisfying \eqref{prob:varphi}. We can then separate variables again in the $x$ direction to get eigenfunctions of the form
\begin{equation}
\Phi(x,y,z) = \cos\left(\frac{m\pi}{L} x \right) \cos\left(\frac{n\pi}{M} z\right) Y(y)
\end{equation}
where $m$ and $n$ are non-negative integers and the function $Y$ satisfies $Y'(0) = 0$, $Y'(R) = \sigma Y(R)$ and
\begin{equation}
\frac{Y''}{Y} = \left(\frac{m\pi}{L}\right)^2 + \left(\frac{n\pi}{M}\right)^2 =: \mu^2.
\end{equation}
It follows that $Y(y) = \cosh(\mu y)$ and the eigenvalue is given by $\sigma = \mu \tanh(\mu R)$. As $m$ or $n$ get large, so does $\mu$, and $\tanh(\mu R)$ converges to $1$ exponentially fast. Hence, $\sigma = \mu + O(\mu e^{-\mu R})$ and the eigenvalues barely depend on $R$. The eigenvalue counting function is then given by
\begin{equation}
N(\sigma) = \left\{(m,n) \in \N_0^2 : \left(\frac{m\pi}{\sigma L}\right)^2 + \left(\frac{n\pi}{\sigma M}\right)^2 < 1\right\} + o(\sigma) = \frac{LM}{4\pi} \sigma^2 + \frac{L+M}{2\pi}\sigma + o(\sigma).
\end{equation}
This last expression comes from estimates on the Gauss circle problem (see \cite{Littlewood} for example). Therefore, the asymptotic behavior of $N(\sigma)$ does not depend on $R$.
\end{rem}

\begin{rem}
We expect that the asymptotic behavior of the sloshing eigenvalues should only depend on a neighborhood of the sloshing surface. Therefore, the results we will show on the asymptotic behavior of $N(\sigma)$ should also be valid in the more general case where $\mathcal{W}$ is a piecewise smooth curve $y = w(x)$ with $w(0) = w(L) = 0$, $w(x) < 0$ for $x \in (0, L)$, and making the same angles $\alpha$ and $\beta$ with $\mathcal{S}$.
\end{rem}

\subsection{Motivation}

The sloshing problem has its origins in the theory of hydrodynamics (see \cite[Chapter 9]{lamb} for example). It describes the oscillations of an ideal fluid on the surface of a container, such as coffee in a cup. Modern results and references on the sloshing problem can be found in \cite{kozlov} and \cite{kuznetsov}.

There has been recent interest into the Steklov problem \eqref{eq:steklov}, see \cite{girouardpolterovich} for a survey on the problem. The Steklov eigenvalues correspond to the eigenvalues of the Dirichlet-to-Neumann map which is often referred to as the voltage-to-current map. It is very closely related to the Calderòn problem \cite{calderon} upon which lies electrical impedance tomography, used in geophysical and medical imaging.

If $\partial\Omega$ and $\rho$ are smooth, the Dirichlet-to-Neumann operator is a pseudodifferential operator and one can use pseudodifferential techniques to study its spectrum \cite{GPPS, awp, PS, rozenblumalmostsim}. However, whenever $\partial\Omega$ is not smooth (in the presence of corners for example), those techniques fail and other approaches have to be considered. The simplest example of $\Omega$ without a smooth-boundary is a cuboid in $\R^n$. The eigenvalue counting function on cuboids has been studied in \cite{cuboids} where they showed that it admits a two-term asymptotic where the second term accounts for the $n-2$ dimensional facets of the cuboid, e.g. the length of the edges in a regular cube. However, in the case of a cuboid, all the angles between the facets are the same right angles. Changing the angles should change the asymptotic and that is what we wish to quantify.

The problem we are considering stems from the work of Levitin, Parnovski, Polterovich and Sher in \cite{LPPS1} and \cite{LPPS2}. In both papers, their goal is to understand how angles inside a two dimensional curvilinear polygon affect its Steklov or sloshing eigenvalues. They started off by considering the same triangles $\Sigma$ as we described in \ref{sec:problem}. Their goal was then to solve
\begin{equation}\label{eq:probLPPS}
\begin{cases}
\Delta u = 0 & \text{in } \Sigma, \\
\del_\nu u = 0 & \text{on } \mathcal{W}, \\
\del_\nu u = \sigma u & \text{on } \mathcal{S}.
\end{cases}
\end{equation}
This problem is exactly like the problem \eqref{prob:varphi} with $n = 0$. They were able to show the following.
\begin{thm}[Levitin, Parnovski, Polterovich, Sher \cite{LPPS1}, 2019]
Suppose that $0 < \alpha \leq \beta < \frac{\pi}{2}$. Then the following asymptotic expansion holds for the eigenvalues of problem \eqref{eq:probLPPS} as $k \rightarrow \infty$:
\begin{equation}
\sigma_k L = \pi\left(k - \frac{1}{2}\right) - \frac{\pi^2}{8}\left(\frac{1}{\alpha} + \frac{1}{\beta}\right) + o(1).
\end{equation}
\end{thm}

A key idea of their proof was to reduce the problem to angles of the form $\frac{\pi}{2q}$ for $q \in \N$, which are refered to as exceptional angles. They then used domain monotonicity to show the result for arbitrary angles $\alpha$ and $\beta$ by bounding them from above and below by exceptional angles. Considering these exceptional angles allowed them to compute explicitly solutions from the sloping beach problem emanating from each corner which they glued together to obtain approximate solutions of \eqref{eq:probLPPS} called quasimodes. Through careful analysis of the quasimodes, they were able to show that the related quasi-eigenvalues were close to real eigenvalues of problem \eqref{eq:probLPPS} and approximated all of them.

We now aim to generalize their approach to three dimensions. By separating variables, we can bring everything back to two dimensions, but we are now solving for solutions of the Helmholtz equation with different eigenvalues $\lambda_n^2$ rather than for harmonic functions.

\subsection{Main results}

Our first result concerns the case where $\alpha = \beta = \frac{\pi}{4}$ and is obtained by finding explicitly the eigenfunctions.
\begin{thm}\label{thm:pi/4}
The eigenvalue counting function of problem \eqref{sloshing} with $\alpha = \beta = \frac{\pi}{4}$ is given by
\begin{equation}
N(\sigma) = \frac{LM}{4\pi}\sigma^2 + \frac{L + M(2\sqrt{2} + 1)}{2\pi}\sigma + o(\sigma).
\end{equation}
\end{thm}

For other values of $\alpha$ and $\beta$, we were not able to find the eigenfunctions explicitly and it probably is unfeasible. Hence, we have to resort to new methods. Our idea is to construct quasimodes that are approximate solutions of problem \eqref{sloshing}. More specifically, our quasimodes will satisfy the eigenvalue condition on the sloshing surface, but rather than satisfy the Neumann condition on the walls, the normal derivative will decay exponentially with respect to their eigenvalue $\sigma$. Hence, the quasimodes will be very close to being eigenfunctions and we should expect the error between quasi-eigenvalues and real eigenvalues of the problem to converge to zero as they get large. We will use two kinds of quasimodes that we refer to as edge waves and surface waves. Their construction is presented in Section \ref{sec:construction}. Let $N^e(\sigma)$ and $N^s(\sigma)$ be the counting functions for the eigenvalues of the edge waves and surface waves respectively. Our main results then concern the asymptotic expansion of those counting functions. Before stating them, we need to introduce some quantities.

Let $\alpha = \frac{\pi}{2q}$ and $\beta = \frac{\pi}{2r}$. Define
\begin{equation}
\theta_\alpha(t) = -\sum_{j=1}^{q-1} \arctan\left(\frac{\sqrt{1-t^2} \sin \frac{j\pi}{q}}{1 - \cos \frac{j\pi}{q}}\right).
\end{equation}
and define similarly $\theta_\beta$ by substituting $q$ by $r$. Furthermore, let $\nu_{\alpha,\beta} = qr \mod{2}$ and $\kappa_{\alpha,\beta}$ be $0$ if $q$ and $r$ share the same parity, and $\frac{1}{2}$ otherwise. Then, we show the following two theorems.

\begin{thm}\label{thm:edge}
The counting function $N^e(\sigma)$ for the edge waves quasi-eigenvalues satisfies the following asymptotic expansion:
\begin{equation}
N^e(\sigma) = \nu_{\alpha,\beta}\frac{M\sigma}{\pi} + \sum_{m=0}^{\floor{\frac{q}{2} - 1}} \frac{M\sigma}{\pi\sin(2m+1)\alpha} + \sum_{\ell=0}^{\floor{\frac{r}{2} - 1}} \frac{M\sigma}{\pi\sin(2\ell+1)\beta} + O(1).
\end{equation}
\end{thm}

\begin{thm}\label{thm:surface}
The counting function $N^s(\sigma)$ for the surface waves quasi-eigenvalues satisfies the following asymptotic expansion:
\begin{equation}
N^s(\sigma) = \frac{LM}{4\pi}\sigma^2 + \frac{L-M}{2\pi} \sigma + \frac{\kappa_{\alpha,\beta} M}{\pi}\sigma + \frac{M\sigma}{\pi^2} \left(\int_0^1 [\theta_\alpha(t) + \theta_\beta(t)] \de t \right) + o(\sigma).
\end{equation}
\end{thm}

Ideally, these quasi-eigenvalues would correspond to the real eigenvalues of the sloshing problem. We will show that for every quasi-eigenvalue, there is a sloshing eigenvalue exponentially close to it. Indeed, if we denote by $\{\tilde{\sigma}_j\}_{j \in \N_0}$ the set of our quasi-eigenvalues arranged in ascending order, then Lemma \ref{lem:LPPS} implies the following.

\begin{lem}\label{lem:weakLPPS}
There exist positive constants $C$ and $c$ such that for every $j \in \N_0$, there exists $k(j) \in \N_0$ such that
\begin{equation}
\abs{\tilde{\sigma}_j - \sigma_{k(j)}} \leq Ce^{-c\tilde{\sigma}_j}.
\end{equation}
\end{lem}

Hence, by showing that all but finitely many values of $k(j)$ can be chosen distinctly, we can show that $N(\sigma)$ is bounded from below by the sum of our quasi-eigenvalue counting functions.

\begin{thm}\label{thm:gg}
The eigenvalue counting function $N(\sigma)$ of problem \eqref{sloshing} satisfies
\begin{equation}
N(\sigma) \geq N^e(\sigma) + N^s(\sigma) + o(\sigma).
\end{equation}
\end{thm}

However, we will not be able to show that there is a quasi-eigenvalue close to every real eigenvalue of the sloshing problem, which would show that $N(\sigma) \leq N^e(\sigma) + N^s(\sigma) + o(\sigma)$. This leads us the conjecture the following.

\begin{conj} \label{conjecture}
The eigenvalue counting function $N(\sigma)$ of problem \eqref{sloshing} is given by
\begin{equation}
N(\sigma) = N^e(\sigma) + N^s(\sigma) + o(\sigma).
\end{equation}
\end{conj}

Note that when $\alpha = \beta = \frac{\pi}{2}$ or $\alpha = \beta = \frac{\pi}{4}$, this coincides with what we got in Remark \ref{rem:pi/2} and what we show in Theorem \ref{thm:pi/4}. Although we are not able to prove Conjecture \ref{conjecture} for other angles, we provide numerical evidence supporting it in Section \ref{sec:num}. As mentioned above, this conjecture hinges on showing that there is a quasi-eigenvalue next to each sloshing eigenvalue. This motivates the next definition and our second conjecture.

\begin{defi}
We say that the sequence of quasi-eigenvalues $\tilde{\sigma}_j$ is {\it asymptotically complete} if we can choose the function $k$ in Lemma \ref{lem:weakLPPS} in a way that there exists integers $N > 0$ and $J \in \Z$, such that for any $j > N$, $k(j) = j + J$.
\end{defi}

This definition is inspired by the similar definition in \cite{LPPS1}, but without the ``quasi-frequency gap'' condition.
 
\begin{conj}\label{conj:complete}
The set of all edge wave and surface wave quasi-eigenvalues is asymptotically complete.
\end{conj}

Note that Conjecture \ref{conj:complete} implies Conjecture \ref{conjecture}. We also support Conjecture \ref{conj:complete} with numerical evidence in Section \ref{sec:num}. A priori, the integer $J$ in the definition of asymptotic completeness can be of any sign. Moreover, it appears from our numerical experiments that the larger $q$ and $r$ are, the larger $J$ gets. Finding the specific value of $J$ is a separate issue, but it is clear that it depends on both angles.

Both our conjectures are only valid for angles of the form $\frac{\pi}{2q}$. At the moment, we are unable to deal with arbitrary angles, see Section \ref{sec:quasidisc}.

\subsection{Our approach}

Firstly, in Section \ref{sec:pi/4}, we compute explicitly the eigenfunctions and eigenvalues for the case where $\alpha = \beta = \frac{\pi}{4}$. From those computations, we show Theorem \ref{thm:pi/4}. Then, in Section \ref{sec:construction}, using solutions coming from the theory of the sloping beach problem, we construct quasimodes for any angles $\alpha = \frac{\pi}{2q}$ and $\beta = \frac{\pi}{2r}$. These solutions arise in two forms that we refer to as edge waves and surface waves, corresponding to the discrete and continuous parts of the spectrum of the sloping beach problem (see \cite{ursell}). Using these quasimodes, we find suitable asymptotic formulas for $N^e$ and $N^s$ in Section \ref{sec:counting}, showing Theorems \ref{thm:edge} and \ref{thm:surface}. Counting the eigenvalues coming from edge wave solutions is straightforward. However, counting the eigenvalues coming from surface wave solutions is more involved and we reduce the problem to that of counting integer points in a randomly perturbed ellipse. We discuss the theory of quasimodes and show Theorem \ref{thm:gg} in Section \ref{sec:discuss}, as well as provide numerical evidence of Conjectures \ref{conjecture} and \ref{conj:complete}.

\subsection{Acknowledgments}

The research of J.M. and C.S. was supported by NSERC's USRA, and was done as part of an intership at Université de Montréal, under the supervision of Iosif Polterovich. The research of S.St-A. was supported by NSERC's CGS-M and FRQNT's M.Sc. scholarship (B1X). This work is part of his M.Sc. studies at the Université de Montréal, under the supervision of Iosif Polterovich. Authors would like to thank Iosif Polterovich for useful discussions and guidance. They would like to thank Zeev Rudnick for the proof of Lemma \ref{lem:exp} and introducing them to the theory of exponential sums. S.St-A. would also like to thank Thomas Davignon and Alexis Leroux-Lapierre for useful discussions, as well as Jean Lagacé, Michael Levitin, Leonid Parnovski and David Sher for their comments.

\section{Explicit computation of the case \texorpdfstring{$\alpha = \beta = \frac{\pi}{4}$}{alpha = beta = pi/4}} \label{sec:pi/4}

Consider the cuboid $\tilde{\Omega} = [-L/2, L/2]^2 \times [0,M]$.

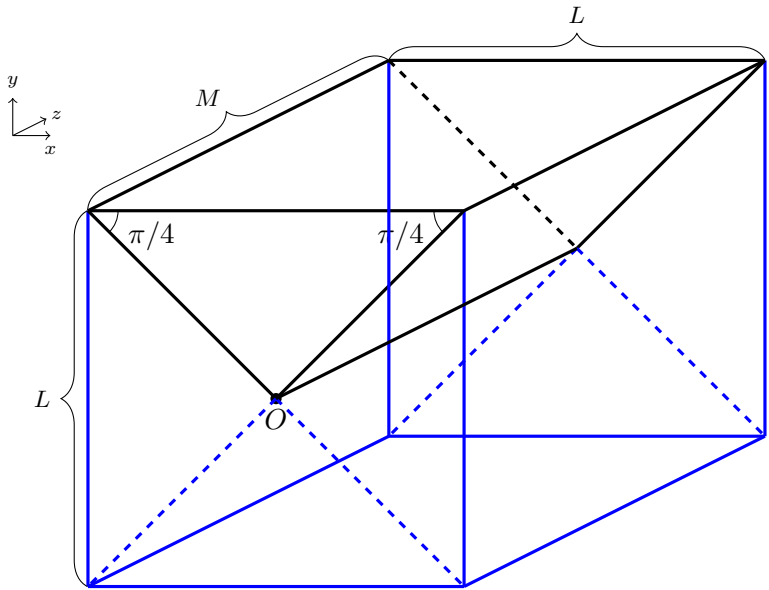
\begin{figure}[!ht]
\begin{center}
\begin{tikzpicture}[scale=0.5]
\coordinate (O) at (-2,2);
\coordinate (X) at (-1,2);
\coordinate (Y) at (-2,3);
\coordinate (Z) at (-1.1056,2.4472);
\coordinate (ZO) at (-1.25,2.5472);

\draw[->] (O) -- (X);
\draw[->] (O) -- (Y);
\draw[->] (O) -- (Z);

\draw (X) node[anchor = north]{\tiny $x$};
\draw (Y) node[anchor = south]{\tiny $y$};
\draw (ZO) node[anchor = west]{\tiny $z$};

\draw (0.8,0) node[anchor=north west]{$\pi/4$} arc (0:-45:0.8);
\draw (9.2,0) node[anchor=north east]{$\pi/4$} arc (180:225:0.8);

\draw (5,-5) node[circle, fill, inner sep=1.5pt]{};
\draw (5,-5) node[anchor = north]{$O$};

%rectangle avant
\draw[very thick] (0,0) -- (10,0);
\draw[very thick,blue] (0,0) -- (0,-10);
\draw[very thick,blue] (0,-10) -- (10,-10);
\draw[very thick,blue] (10,-10) -- (10,0);

%triangle avant
\draw[very thick] (0,0) -- (5,-5);
\draw[very thick] (5,-5) -- (10,0);
\draw[very thick,blue,dashed] (5,-5) -- (0,-10);
\draw[very thick,blue,dashed] (5,-5) -- (10,-10);

%rectangle arrière
\draw[very thick,black] (8,4) -- (18,4);
\draw[very thick,blue] (8,4) -- (8,-6);
\draw[very thick,blue] (8,-6) -- (18,-6);
\draw[very thick,blue] (18,-6) -- (18,4);

%triangle arrière
\draw[very thick,black,dashed] (8,4) -- (13,-1);
\draw[very thick,black] (13,-1) -- (18,4);
\draw[very thick,blue,dashed] (13,-1) -- (8,-6);
\draw[very thick,blue,dashed] (13,-1) -- (18,-6);

%prisme
\draw[very thick, black] (0,0) -- (8,4);
\draw[very thick, black] (10,0) -- (18,4);
\draw[very thick, black] (5,-5) -- (13,-1);
\draw[very thick,blue] (0,-10) -- (8,-6);
\draw[very thick,blue] (10,-10) -- (18,-6);

\draw [decorate,decoration={brace,amplitude=10pt}, xshift=0pt, yshift=0.4pt] (8,4) -- (18,4) node [black,midway,yshift=0.6cm] {\footnotesize $L$};
\draw [decorate,decoration={brace,amplitude=10pt}, xshift=-0.2pt, yshift=0.2pt] (0,0) -- (8,4) node [black,midway,xshift=-0.4cm,yshift=0.5cm] {\footnotesize $M$};
\draw [decorate, decoration={brace,amplitude=10pt}, xshift=-0.4pt, yshift=0pt] (0,-10) -- (0,0) node [black,midway,xshift=-0.6cm] {\footnotesize $L$};
\end{tikzpicture}
\end{center}
\caption{Reflections of $\Omega$ along $\Gamma_N$ to get $\tilde{\Omega}$.}
\label{cube}
\end{figure}

Let $\tilde{\Gamma}_S \subset \partial \tilde{\Omega}$ denote the four faces of the cuboid with area $LM$ and let $\tilde{\Gamma}_N \subset \partial \tilde{\Omega}$ denote the two faces of the cuboid with area $L^2$. If $\Phi : \Omega \rightarrow \R$ is a solution of \eqref{sloshing}, then the function $\tilde{\Phi} : \tilde{\Omega} \rightarrow \R$ obtained by reflecting evenly $\Phi$ along a rectangular part of $\Gamma_N$ three times satisfies
\begin{equation}\label{sloshingcube}
\begin{cases}
\Delta \tilde{\Phi} = 0 & \text{in } \tilde{\Omega}, \\
\del_\nu\tilde{\Phi} = 0 & \text{on } \tilde{\Gamma}_N, \\
\del_\nu\tilde{\Phi} = \sigma \tilde{\Phi} & \text{on } \tilde{\Gamma}_S.
\end{cases}
\end{equation}
We illlustrate these reflections in Figure \ref{cube} (note that we changed the position of the origin $O$ from Figure \ref{fig:domain}). Conversely, if $\tilde{\Phi}$ is a solution of \eqref{sloshingcube} that is symmetric along both planes spanned by the rectangular parts of $\Gamma_N$, then $\Phi = \tilde{\Phi}|_{\Omega}$ is a solution of \eqref{sloshing}. Therefore, solving \eqref{sloshing} is equivalent to finding solutions with even symmetries along these planes. In other words, the functions must be invariant under the change of variables $(x,y) \mapsto (y,x)$ and $(x,y) \mapsto (-y,-x)$. Finding such solutions is much easier since we can separate variables completely.

Let $\lambda_n = \frac{n\pi}{M}$ for $n \in \N_0$. The corresponding eigenfunctions then take the form
\begin{equation}
\tilde{\Phi}(x,y,z) = \varphi(x,y) \cos(\lambda_n z)
\end{equation}
where $\varphi(x,y)$ is given by one of the functions in Table \ref{table:eigen}. One can check that all these eigenfunctions satisfy $\varphi(x,y) = \varphi(y,x) = \varphi(-y,-x)$. 

\begin{center}
\begin{table}[!h]
\centering
\begin{tabular}{c c c}
Eigenfunction $\varphi$  & Conditions on $\chi$ and $n$ & Eigenvalue \\ \toprule 
$\cosh\left(\frac{\lambda_n}{\sqrt{2}}x\right) \cosh\left(\frac{\lambda_n}{\sqrt{2}}y\right)$ & $n \geq 0$ & $\frac{\lambda_n}{\sqrt{2}}\tanh\left(\frac{\lambda_n}{2\sqrt{2}}L\right)$\\
\midrule
$\sinh\left(\frac{\lambda_n}{\sqrt{2}}x\right) \sinh\left(\frac{\lambda_n}{\sqrt{2}}y\right)$ & $n > 0$ & $\frac{\lambda_n}{\sqrt{2}}\coth\left(\frac{\lambda_n}{2\sqrt{2}}L\right)$\\
\midrule
$\begin{aligned}
&\cos(\chi x)\cosh(\sqrt{\chi^2 + \lambda_n^2} y) \\ &+ \cos(\chi y)\cosh(\sqrt{\chi^2 + \lambda_n^2} x)
\end{aligned}$ & 
$\begin{gathered}
n \geq 0 \\
-\chi\tan \frac{\chi L}{2} = \sqrt{\chi^2 + \lambda_n^2}\tanh \left(\sqrt{\chi^2 + \lambda_n^2}\frac{L}{2}\right)
\end{gathered}$ & $-\chi\tan \frac{\chi L}{2}$\\
\midrule
$\begin{aligned}
&\sin(\chi x)\sinh(\sqrt{\chi^2 + \lambda_n^2} y) \\ &+ \sin(\chi y)\sinh(\sqrt{\chi^2 + \lambda_n^2} x)
\end{aligned}$ & 
$\begin{gathered}
n \geq 0 \\
\chi\cot \frac{\chi L}{2} = \sqrt{\chi ^2 + \lambda_n^2}\coth\left( \sqrt{\chi^2 + \lambda_n^2}\frac{L}{2}\right)
\end{gathered}$ & $\chi\cot \frac{\chi L}{2}$ \\
\midrule
$xy$ & $n = 0$ & $\frac{2}{L}$ \\
\bottomrule \\
\end{tabular}
\caption{Eigenfunctions $\varphi(x,y)$ obtained by separation of variables that are symmetric with respect to $y=x$ and $y=-x$.} \label{table:eigen}
\end{table}
\end{center}

Let $N^{(i)}(\sigma)$ be the number of eigenvalues of problem \ref{sloshing} smaller than $\sigma$ corresponding to eigenfunctions in the $i$-th line of Table \ref{table:eigen} for $i = 1, \dots, 5$. First, since there is only one function of type $5$, $N^{(5)}(\sigma) = O(1)$. Second, since the hyperbolic tangents and cotangents quickly converge to $1$, we have
\begin{equation}
N^{(1)}(\sigma) = N^{(2)}(\sigma) = \frac{\sqrt{2} M}{\pi} \sigma + O(1).
\end{equation}
We can rewrite the third condition on $\chi$ and $n$ as
\begin{equation}\label{eq:cond3}
\chi = \frac{\pi}{L} \left(-\frac{2}{\pi}\arctan\left[\sqrt{1 + (\lambda_n/\chi)^2} \tanh\left( \sqrt{\chi^2 + \lambda_n^2}\frac{L}{2}\right)\right] + 2m\right)
\end{equation}
for $m \in \N_0$. Similarly, the fourth condition is given by
\begin{equation}\label{eq:cond4}
\chi = \frac{\pi}{L}\left(-\frac{2}{\pi}\arctan\left[\sqrt{1 + (\lambda_n/\chi)^2} \coth\left( \sqrt{\chi^2 + \lambda_n^2}\frac{L}{2}\right)\right] + (2m+1)\right)
\end{equation}
where again $m \in \N_0$. We only consider the positive solutions of $\chi$ as the negative solutions give rise to the same eigenfunctions. When $m = 0$, equation \eqref{eq:cond3} admits no solution $\chi > 0$. Notice that the hyperbolic tangents and cotangents quickly converge to $1$ as $\sigma = \sqrt{\chi^2 + \lambda_n^2} + O(e^{-\sigma})$ gets big, and hence the solutions of both equations \eqref{eq:cond3} and \eqref{eq:cond4} are exponentially close to the solutions of
\begin{equation}
\chi = \frac{\pi}{L}\left(m -\frac{2}{\pi}\arctan\sqrt{1 + (\lambda_n/\chi)^2} \right)
\end{equation}
for $m \in \N$. The eigenvalues are given by $\sigma = \sqrt{\chi^2 + \lambda_n^2} + O(e^{-\sigma})$ and so
\begin{equation}
\sigma^2 = \left(\frac{\left(m - \frac{2}{\pi}\arctan\sqrt{1 + (\lambda_n/\chi)^2}\right)\pi}{L}\right)^2 + \left(\frac{n\pi}{M}\right)^2 + O(e^{-\sigma}).
\end{equation}
Moreover, we have
\begin{equation}
\arctan\sqrt{1 + (\lambda_n/\chi)^2} = -\arctan\sqrt{1 - (\lambda_n/\sigma)^2} + \frac{\pi}{2} + O(e^{-\sigma}).
\end{equation}
By plugging this relation into the previous equation and including the $\frac{\pi}{2}$ into the integer $m$, it follows that the eigenvalues $\sigma$ of type $3$ and $4$ are exponentially close to the solutions of
\begin{equation}\label{eq:sigma}
\sigma^2 = \left(\frac{\left(m + \frac{2}{\pi}\arctan\sqrt{1 - (\lambda_n/\sigma)^2}\right)\pi}{L}\right)^2 + \left(\frac{n\pi}{M}\right)^2
\end{equation}
for $m \geq 0$ and $n \geq 0$. In Section \ref{sec:counting}, we show how to count the number of solutions of such an equation. Theorem \ref{thm:pi/4} then follows from those calculations.

It is important to note the behavior of the eigenfunctions in Table \ref{table:eigen}. We can ignore the singular solution $xy$ since it doesn't contribute significantly to $N(\sigma)$. The first two functions are concentrated in the corners of the square $[-L/2,L/2]^2$. Hence, the corresponding solutions $\Phi$ on $\Omega$ are concentrated on the edges of the sloshing surface that have length $M$. It makes sense to call such solutions {\it{edge waves}}. On the other hand, the third and fourth solutions are concentrated on the edge of the square $[-L/2,L/2]^2$ where they oscillate. Therefore, the corresponding solutions $\Phi$ on $\Omega$ oscillate on the whole sloshing surface, but vanish fast inside $\Omega$. In contrast to the edge waves, we refer to those solutions as {\it{surface waves}}.

Hence, in order to approximate solutions on a domain $\Omega$ with angles $\alpha = \frac{\pi}{2q}$ and $\beta = \frac{\pi}{2r}$, we have to consider both kinds of waves. In the next section, we show how to construct these solutions for each type of wave.

\section{Construction of quasimodes}\label{sec:construction}

In order to approximate solutions of the sloshing problem, we are going to glue together solutions of a similar problem emanating from both corners. The functions we obtain are not exactly eigenfunctions for our problem. Nonetheless, they give rise to eigenvalues that should be close to the actual eigenvalues. We refer to them as quasi-eigenvalues. We discuss the theory of quasimodes in Section \ref{sec:discuss}. The functions we use arise from the solutions of the sloping beach problem which has both discrete and continuous spectrum (see \cite{ursell} and \cite{evans}). We construct quasimodes coming from both parts of the spectrum. We refer to the solutions corresponding to the discrete part of the spectrum as {\it edge waves} since they will generate quasimodes concentrated on the edges of the prism $\Omega$. In analogy, we refer to the solutions corresponding to the continuous part of the spectrum as {\it surface waves} since the resulting quasimodes will oscillate on the whole sloshing surface $\Gamma_S$ and decay exponentially inside $\Omega$. Lemmas \ref{lem:edge} and \ref{lem:valpha} will confirm the behaviors of the edge wave and surface wave quasimodes respectively.

Note that although the spectrum corresponding to surface waves is continuous, the resulting quasi-eigenvalues will be discrete, since we will get ``gluing'' conditions in order for our resulting approximate solutions to be sufficiently smooth.

\subsection{Sloping beach problem}

Consider the angular sector $S_\alpha = \{-\alpha \leq \theta \leq 0\}$ in the $xy$-plane as illustrated in Figure \ref{fig:angular} and let $\Omega_\alpha = S_\alpha \times [0,M]$ be a sloping beach domain.
\begin{figure}[!ht]
\begin{center}
\begin{tikzpicture}[scale=0.7]
\draw[black,->] (-2,0) -- (8,0) node[anchor=south]{x};
\draw[black,->] (0,-8) -- (0,2) node[anchor=east]{y};

\draw (0.8,0) node[anchor=north west]{\small $\alpha$} arc (0:-45:0.8);

\draw (5.8,-2) -- cycle node[anchor=north west]{\large $S_{\alpha}$};

\draw[very thick,red] (0,0) -- node[anchor=south]{$I_1$} (7.8,0);
\draw[very thick,blue] (0,0) -- node[anchor=north east]{$I_2$} (7.8,-7.8);
\end{tikzpicture}
\end{center}
\caption{The angular sector $S_\alpha$.} \label{fig:angular}
\label{plage}
\end{figure}
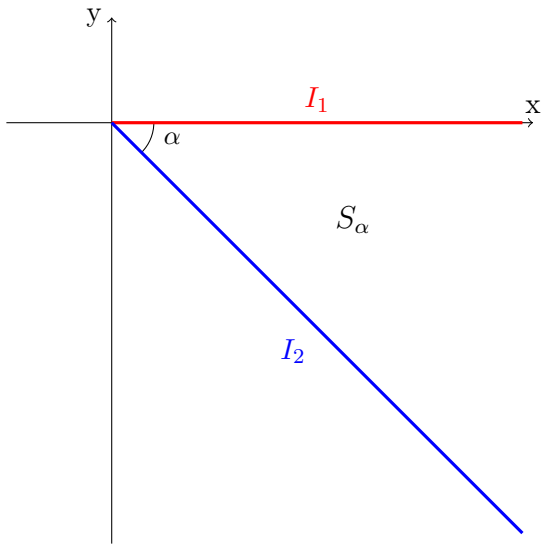
The water surface is given by $I_S = I_1 \times [0,M]$ and the bottom of the beach is given by $I_N = I_2 \times [0,M]$ where $I_1 = \{\theta = 0\}$ and $I_2 = \{\theta = - \alpha\}$. The sloping beach problem corresponds to finding a velocity potential $\Phi : \Omega_\alpha \rightarrow \R$ such that $\Phi$ is harmonic inside $\Omega_\alpha$, satisfies Neumann boundary conditions on $I_N$ and the Steklov boundary condition $\del_\nu \Phi = \sigma \Phi$ on $I_S$. By separating variables, we get that $\Phi = \varphi(x,y) \cos\lambda_n z$ with $\lambda_n = \frac{n\pi}{M}$ and $\varphi$ satisfying
\begin{equation}\label{eq:beachprob}
\begin{cases}
\Delta \varphi = \lambda_n^2 \varphi & \text{in } S_\alpha, \\
\del_\nu \varphi = 0 & \text{on } I_2, \\
\del_\nu \varphi = \sigma \varphi & \text{on } I_1.
\end{cases}
\end{equation}

We will create an approximate solution of \eqref{prob:varphi} by using solutions from the sloping beach problem \eqref{eq:beachprob} coming from each corner of $\Sigma$. These solutions will need to meet smoothly and give rise to the same eigenvalue. This gluing condition will then determine the possible quasi-eigenvalues.

\subsection{Edge wave solutions of the sloping beach problem}

Let $0 < \alpha \leq \frac{\pi}{2}$ and $n \in \N_0$. The edge wave solutions of the sloping beach problem \eqref{eq:beachprob} given by Ursell \cite{ursell} are as follows. For $0 \leq m \leq \frac{\pi}{4\alpha} - \frac{1}{2}$, $m \in \Z$, let
\begin{align} \label{eq:edge}
\varphi_{nm}(x,y) &= e^{-\lambda_n( x \cos\alpha - y \sin\alpha)} + \sum_{j=1}^m A_{jm}\left(e^{-\lambda_n (x \cos(2j-1)\alpha + y \sin(2j-1)\alpha)} + \right. \\ &\left. \qquad+ e^{-\lambda_n( x \cos(2j+1)\alpha - y \sin(2j+1)\alpha)}\right)
\end{align}
where $A_{jm} = (-1)^j \prod_{r=1}^j \frac{\tan(m - r + 1)\alpha}{\tan(m+r)\alpha}$. One can check that $\varphi_{nm}$ solves \eqref{eq:beachprob} with
\begin{equation}\label{eq:edgevalues}
\sigma_{nm} = \lambda_n \sin (2m+1) \alpha.
\end{equation}
Note that if $n=0$, we get the constant solution and we can ignore it. In other words, there are no edge waves in the two-dimensional sloshing problem. We are particularly interested in the case where $\alpha = \frac{\pi}{2q}$, in which case $\frac{\pi}{4\alpha} - \frac{1}{2} = \frac{q-1}{2}$. In order to study these solutions, we need some estimates on $\varphi_{nm}$ and its derivatives.

\begin{lem}\label{lem:edge}
Let $\alpha = \frac{\pi}{2q}$ for an integer $q \geq 1$. There exist positive constants $C$ and $c$ such that the following estimates hold for all $(x,y)$ in $S_\alpha$.
\begin{enumerate}
\item[1.] For $0 \leq m < \frac{q-1}{2}$,
\begin{equation}\label{eq:estimate1}
\abs{\varphi_{nm}(x,y)} \leq Ce^{-c\lambda_n x}
\end{equation}
and
\begin{equation}\label{eq:estimate2}
\abs{\nabla_{(x,y)} \varphi_{nm}(x,y)} \leq C\lambda_n e^{-c\lambda_n x};
\end{equation}
\item[2.] If $q$ is odd, then for $m = \frac{q-1}{2}$,
\begin{equation}\label{eq:estimate3}
\abs{\varphi_{nm}(x,y) - A_{mm}e^{\lambda_n y}} \leq Ce^{-c\lambda_n x}
\end{equation}
and
\begin{equation}\label{eq:estimate4}
\abs{\nabla_{(x,y)} \left(\varphi_{nm}(x,y) - A_{mm} e^{\lambda_n y}\right)} \leq C\lambda_n e^{-c\lambda_n x}.
\end{equation}
\end{enumerate}
\end{lem}

\begin{proof}
We will abuse notation slightly when using $C$ and $c$ throughout the proof, but they will always denote positive constants depending only on the angle $\alpha$.

The first estimate \eqref{eq:estimate1} will follow from showing that for each exponential in \eqref{eq:edge}, the same estimate holds. Since $y \leq 0$ and $0 \leq (2j+1)\alpha < \frac{\pi}{2}$ for $j < \frac{\pi}{4\alpha}$, the estimate clearly holds for the first and third terms in \eqref{eq:edge}. It remains to show that for all $1 \leq j < \frac{q-1}{2}$,
\begin{equation}\label{eq:Cx}
x\cos(2j-1)\alpha + y\sin(2j-1)\alpha \geq cx
\end{equation}
for some $c > 0$. The condition on $j$ obviously only makes sense as long as $q \geq 4$. We can rewrite $(2j-1)\alpha = \frac{j\pi}{q} - \frac{\pi}{2q}$ as $\frac{\pi}{2} - \frac{(q - 2j + 1)\pi}{2q}$ to get
\begin{equation}
x\cos(2j-1)\alpha + y\sin(2j-1)\alpha = x\sin \frac{(q - 2j + 1)\pi}{2q} + y\cos\frac{(q-2j+1)\pi}{2q}.
\end{equation}
Since $y \in S_\alpha$, we have $-x \tan\alpha \leq y \leq 0$ and hence
\begin{align}
x\cos(2j-1)\alpha + y\sin(2j-1)\alpha &\geq x\sin \frac{(q - 2j + 1)\pi}{2q} - x\cos\frac{(q - 2j + 1)\pi}{2q}\tan\frac{\pi}{2q} \\
&\geq \left(\sin\frac{(q - 2j + 1)\pi}{2q} - \sin\frac{\pi}{2q}\right) x
\end{align}
where we used that $\cos \frac{(q-2j+1)}{2q} < \cos \frac{\pi}{2q}$ since $q - 2j \geq 1$. The constant before $x$ in that last expression is strictly positive and therefore \eqref{eq:Cx} holds.

The estimate \eqref{eq:estimate2} follows from \eqref{eq:estimate1} since differentiating each term in \eqref{eq:edge} with respect to $x$ or $y$ introduces only a factor of at most $\lambda_n$.

Now if $q$ is odd and $m = \frac{q-1}{2}$, we have $(2m+1)\alpha = \frac{\pi}{2}$ and hence the last term in the sum defining $\varphi_{nm}$ is given by $A_{mm} e^{\lambda_n y}$. By the previous calculations, all the other terms satisfy similar estimates to $\eqref{eq:estimate1}$ and $\eqref{eq:estimate2}$. The sum of those terms is precisely $\varphi_{nm}(x,y) - A_{mm} e^{\lambda_n y}$, and hence both \eqref{eq:estimate3} and \eqref{eq:estimate4} hold.
\end{proof}

\subsection{Edge wave quasimodes}

We use the edge wave solutions of the sloping beach problem to construct solutions for the sloshing problem. To do so, we aim to glue together solutions coming from each corner. Notice that if a solution vanishes quickly outside its corresponding corner, we don't need to glue a solution coming from the other corner since it'll simply correspond to the zero solution near the other corner. However, if a solution does not vanish, then we have to be careful since there might not be a solution coming from the other corner for that eigenvalue.

Let $\alpha = \frac{\pi}{2q}$ and $\beta = \frac{\pi}{2r}$. Denote by $\varphi^{\alpha}_{nm}$ (respectively $\varphi_{nm}^\beta$) the edge wave solution of the sloping beach problem coming from angle $\alpha$ with eigenvalue $\sigma^{\alpha}_{nm} = \lambda_n \sin(2m + 1)\alpha$ for $0 \leq m < \frac{q-1}{2}$.

If $q$ is even, every $\varphi_{nm}^\alpha(x,y)$ vanishes exponentially fast outside the corner $\alpha$ by Lemma \ref{lem:edge}, and therefore we can consider them as quasimodes individually. The same applies if $r$ is even for the solutions coming from angle $\beta$ that are given in $\Sigma$ by $\varphi_{n\ell}^\beta(L-x,y)$ with eigenvalue $\sigma^{\beta}_{n\ell} = \lambda_n \sin(2\ell + 1) \beta$ for $0 \leq \ell < \frac{r-1}{2}$.

If $q$ is odd, then as above the solution $\varphi^{\alpha}_{nm}$ is a valid quasimode as long as $m \neq \frac{q-1}{2}$. However, when $m = \frac{q-1}{2}$, by Lemma \ref{lem:edge}, the solution $\varphi_{nm}^{\alpha}$ tends to $A_{mm}^\alpha$ on the surface $y = 0$ with a corresponding eigenvalue $\lambda_n$. In order to get a valid quasimode, there should be a non-zero solution coming from the corner $\beta$ with the same eigenvalue. This is only possible if $r$ is also odd. In that case, we consider the quasimode
\begin{equation}\label{eq:edgepsi}
\psi_n(x,y) = A_{\ell\ell}^\beta \varphi_{nm}^\alpha(x,y) + A_{mm}^\alpha\varphi_{n\ell}^\beta(L-x,y) - A_{mm}^\alpha A_{\ell\ell}^\beta e^{\lambda_n}
\end{equation}
where $\ell = \frac{r-1}{2}$. The last term is present so that we can control $\abs{\del_\nu \psi_n}$ on $\mathcal{W}$. We will use a similar trick for the surface wave quasimodes.

In short, given $n \in \N$, we constructed $\floor{\frac{q}{2}}$ and $\floor{\frac{r}{2}}$ quasimodes coming from the corners $\alpha$ and $\beta$ respectively, as well as an additional quasimode if both $q$ and $r$ are odd.

\begin{rem}
Interestingly, our resulting edge wave quasimodes on the whole domain $\Omega$ oscillate only along the edges of length $M$, but not those of length $L$. The computations of Section \ref{sec:pi/4} confirm that this phenomenon occurs when $\alpha = \beta = \frac{\pi}{4}$. It should also hold for all the other triangular prisms and is motivated by the fact that the sloping beach problem has a single edge wave solution when $\alpha = \frac{\pi}{2}$ given by $e^{\lambda_n y}$, which is constant along the sloshing edge $I_1$. Hence, one could expect there to be solutions oscillating along an edge of length $L$ if the wall adjacent to it met the sloshing surface at an angle smaller than $\frac{\pi}{2}$.
\end{rem}

\subsection{Surface wave solutions of the sloping beach problem}\label{sec:surfacesloping}

Let us now construct surface wave solutions of the sloping beach problem. To do so, we generalize the method used in \cite{LPPS1}. By rescaling in the $z$ variable and by setting $\mu := \lambda_n/\sigma$, the problem \eqref{eq:beachprob} is equivalent to solving
\begin{equation}\label{beachsector}
\begin{cases}
\Delta \varphi = \mu^2 \varphi & \text{in } S_\alpha, \\
\del_\nu \varphi = 0 & \text{on } I_2, \\
\del_\nu \varphi = \varphi & \text{on } I_1.
\end{cases}
\end{equation}
However, recall that we are still solving to find the possible values of $\sigma$ and although it doesn't appear in the last formulation, it is actually hidden in $\mu$.

Let $\xi = -\pi/q$, and for $a,b \in \R$, let $g_{a,b}$ denote the function
\begin{equation}
g_{a,b}(x,y) = e^{x\cos(a) + y\sin(a)}e^{i\sqrt{1-\mu^2}(x\cos(b) + y\sin(b))}.
\end{equation}
We define the linear operators $\mathcal{A}$ and $\mathcal{B}$ by
\begin{equation}
(\mathcal{A}g_{a,b})(x,y) := e^{x\cos(-a + \xi) + y\sin(-a + \xi)}e^{i\sqrt{1-\mu^2}(x\cos(-b + \xi) + y\sin(-b + \xi))} = g_{-a + \xi, -b + \xi}(x,y)
\end{equation}
and
\begin{equation}
(\mathcal{B}g_{a,b})(x,y) := C_{a,b}e^{x\cos(a) - y\sin(a)}e^{i\sqrt{1-\mu^2}(x\cos(b) - y\sin(b))} = C_{a,b} g_{-a,-b}(x,y)
\end{equation}
where
\begin{equation}
C_{a,b} = \frac{\sin a + i\sqrt{1-\mu^2} \sin b - 1}{\sin a + i\sqrt{1-\mu^2}\sin b + 1}.
\end{equation}
For an arbitrary function $u$ on $S_\alpha$, we define its Steklov defect by
\begin{equation}
\mathrm{SD}(u) := \left.\left(\del_\nu u - u\right)\right|_{I_1}.
\end{equation}
Note that $\mathrm{SD}(u) = 0$ if and only if $u$ satisfies the Steklov condition on $I_1$ with eigenvalue $1$. By simple calculations, one can show that these operators have the following useful properties.

\begin{prop}\label{propAB}
Let $g$ be as above. We have
\begin{enumerate}
\item $(g - \mathcal{A}g)|_{I_2} = 0$,
\item $\del_\nu(g + \mathcal{A}g)|_{I_2} = 0$,
\item $\mathrm{SD}(g + \mathcal{B} g) = 0$.
\end{enumerate}
\end{prop}

We will use these properties to construct a suitable function on $S_\alpha$. Let $f_0(x,y) = e^y e^{-i\sqrt{1-\mu^2}x}$, i.e. $f_0$ is given by $g_{\frac{\pi}{2}, \pi}$. For $1 \leq m \leq 2q-1$, we construct the functions
\begin{equation}
f_m =
\begin{cases}
\mathcal{A}f_{m-1} & \text{if } m \text{ is odd}, \\
\mathcal{B}f_{m-1} & \text{if } m \text{ is even}.
\end{cases}
\end{equation}
Finally, we let
\begin{equation}
v_\alpha = \sum_{m=0}^{2q-1} f_m.
\end{equation}
The function $v_\alpha$ is our main interest. In fact, it is a solution of \eqref{beachsector}!

\begin{thm}
The function $v_\alpha$ as defined above satisfies $\Delta v_\alpha = \mu^2 v_\alpha$ in $S_\alpha$, the Neumann condition on $I_2$ and $\mathrm{SD}(v_\alpha) = 0$. In other words, it is a solution of \eqref{beachsector}.
\end{thm}

\begin{proof}
First off, we can see that for any choice of $a,b \in \R$, we have
\begin{equation}
\Delta g_{a,b} = \left[\mu^2 + 2i\sqrt{1 - \mu^2} \cos(a - b)\right] g_{a,b}.
\end{equation}
Since $f_0 = g_{\frac{\pi}{2}, \pi}$, we have $\Delta f_0 = \mu^2 f_0$. Both $\mathcal{A}$ and $\mathcal{B}$ act on $g_{a,b}$ by scaling and modifying the coefficients $a$ and $b$, but keep the value of $\abs{a-b}$ unchanged. Then since $f_m$ is obtained by consecutively applying $\mathcal{A}$ and $\mathcal{B}$ on $f_0$, we also have $\Delta f_m = \mu^2 f_m$ for all $m$. By linearity, it then follows that $\Delta v_\alpha = \mu^2 v_\alpha$.

For the Neumann condition, we see that we can write $v_\alpha$ as
\begin{equation}
v_\alpha = \sum_{m=0}^{q-1} (f_{2m} + f_{2m+1}) = \sum_{m=0}^{q-1} (f_{2m} + \mathcal{A}f_{2m})
\end{equation}
and therefore, by Proposition \ref{propAB},
\begin{equation}
\left.\del_\nu v_\alpha\right|_{I_2} = \sum_{m=0}^{q-1} \del_\nu(f_{2m} + \mathcal{A}f_{2m})|_{I_2} = 0.
\end{equation}

It remains to show that $\mathrm{SD}(v_\alpha) = 0$. We now write $v_\alpha$ as
\begin{equation}
v_\alpha = f_0 + \sum_{m=1}^{q-1} (f_{2m-1} + f_{2m}) + f_{2q-1} = f_0 + \sum_{m=1}^{q-1} (f_{2m-1} + \mathcal{B}f_{2m-1}) + f_{2q-1}
\end{equation}
and therefore, by Proposition \ref{propAB} and linearity of the Steklov defect,
\begin{equation}
\mathrm{SD}(v_\alpha) = \mathrm{SD}(f_0) + \mathrm{SD}(f_{2q-1}).
\end{equation}
Since $f_0 = e^y e^{-i\sqrt{1-\mu^2}x}$, we easily see that $\mathrm{SD}(f_0) = 0$. Let us now show that $\mathrm{SD}(f_{2q-1}) = 0$. For any choice of $a$ and $b$,
\begin{equation}
(\mathcal{BA})g_{a,b} = C_{-a+\xi,-b+\xi} g_{a-\xi, b-\xi}.
\end{equation}
Hence, since $f_{2q-1} = \mathcal{A}(\mathcal{BA})^{q-1} f_0 = \mathcal{A}(\mathcal{BA})^{q-1} g_{\frac{\pi}{2}, \pi}$, we get
\begin{equation}\label{fgamma}
f_{2q-1} = \mathcal{A}\left[\left(\prod_{j=1}^{q-1} C_{-\frac{\pi}{2} + j\xi, -\pi + j\xi}\right) g_{\frac{\pi}{2}-(q-1)\xi, \pi - (q-1)\xi}\right] = \gamma(\xi) g_{q\xi - \frac{\pi}{2}, q\xi - \pi}
\end{equation}
where
\begin{equation}
\gamma(\xi) := \prod_{j=1}^{q-1} C_{-\frac{\pi}{2} + j\xi, -\pi + j\xi}.
\end{equation}
Since $\xi = -\pi/q$, we get $f_{2q-1} = \gamma(\xi) g_{-\frac{3\pi}{2}, -2\pi} = \gamma(\xi) e^y e^{i\sqrt{1-\mu^2}x}$ and thus $\mathrm{SD}(f_{2q-1}) = 0$. It follows that $\mathrm{SD}(v_\alpha) = 0$.
\end{proof}

In the previous proof, we started to compute $f_{2q-1}$. Moving forward, we will need its exact expression.

\begin{lem}
The function $f_{2q-1}$ is given by $\gamma(\xi) e^y e^{i\sqrt{1-\mu^2}x}$ where
\begin{equation}
\gamma(\xi) = (-1)^{q-1} \exp\left[2i \sum_{j=1}^{q-1} \arctan\left(\frac{\sqrt{1-\mu^2}\sin\frac{j\pi}{q})}{\cos\frac{j\pi}{q} - 1}\right)\right].
\end{equation}
\end{lem}

\begin{proof}
The expression of $f_{2q -1}$ follows from \eqref{fgamma}. Moreover,
\begin{align}
\gamma(\xi) &= \prod_{j=1}^{q-1} \frac{\sin(-\frac{\pi}{2} - \frac{j\pi}{q}) + i\sqrt{1-\mu^2} \sin (-\pi - \frac{j\pi}{q}) - 1}{\sin (-\frac{\pi}{2} - \frac{j\pi}{q}) + i\sqrt{1-\mu^2}\sin (-\pi - \frac{j\pi}{q}) + 1} \\
&= \prod_{j=1}^{q-1} \frac{-\cos \frac{j\pi}{q} + i\sqrt{1-\mu^2}\sin \frac{j\pi}{q} - 1}{- \cos \frac{j\pi}{q} + i \sqrt{1-\mu^2}\sin\frac{j\pi}{q} + 1} \\
&= (-1)^{q-1} \prod_{j=1}^{q-1} \frac{\cos\frac{j\pi}{q} + i\sqrt{1-\mu^2} \sin \frac{j\pi}{q} - 1}{\cos\frac{j\pi}{q} - i\sqrt{1-\mu^2}\sin \frac{j\pi}{q} - 1}
\end{align}
where we have reordered the terms in the numerator by $j \mapsto q-j$ to get the last expression. The denominator is the complex conjugate of the numerator. Therefore, $\abs{\gamma(\xi)} = 1$ and
\begin{equation}
\arg \gamma(\xi) = (q-1)\pi + \sum_{j=1}^{q-1} 2\arctan\left(\frac{\sqrt{1-\mu^2}\sin\frac{j\pi}{q}}{\cos\frac{j\pi}{q} - 1}\right).
\end{equation}
The claim readily follows.
\end{proof}

\begin{lem}\label{lem:valpha}
There exist positive constants $C$ and $c$ such that for all $(x,y) \in S_\alpha$,
\begin{equation}
v_\alpha(x,y) = e^y e^{-i\sqrt{1-\mu^2}x} + \gamma(\xi) e^y e^{i\sqrt{1-\mu^2}x} + v_\alpha^{\mathrm{d}}(x,y),
\end{equation}
with
\begin{equation}\label{eq:Ralpha}
\abs{v_\alpha^{\mathrm{d}}(x,y)} + \abs{\nabla_{(x,y)} v_\alpha^{\mathrm{d}}(x,y)} \leq Ce^{-cx}.
\end{equation}
In particular, on the boundary $I_1$ the solution $v_\alpha (x,y)$ takes the form
\begin{equation}
v_\alpha(x) = e^{-i\sqrt{1-\mu^2}x} + \gamma(\xi)e^{i\sqrt{1-\mu^2}x} + \text{decaying exponentials}.
\end{equation}
\end{lem}

\begin{proof}
As in the proof of Lemma \ref{lem:edge}, we will abuse notation throughout the proof when using $C$ and $c$, but they again denote positive constants depending only on the angle $\alpha$.

The function $v_\alpha^{\mathrm{d}}$ is given by $\sum_{m=1}^{2q-2} f_m$. Therefore, it suffices to show that each of these $f_m$ satisfies the same estimate as \eqref{eq:Ralpha}. For each such $f_m$,
\begin{equation}
\abs{f_m} = F_m e^{x\cos(a) + y\sin(a)}
\end{equation}
for some constant $F_m > 0$ and $a = \pm(\frac{\pi}{2} + j \frac{\pi}{q})$ where $j \in \{1, \dots, q-1\}$, resulting from the successive applications of $\mathcal{A}$ and $\mathcal{B}$. By periodicity, it is equivalent that $a$ either takes values of the form $\frac{\pi}{2} + \frac{j\pi}{q}$ or $\frac{3\pi}{2} - \frac{j\pi}{q}$ for $1 \leq j \leq \floor{\frac{j}{2}}$. If $a = \frac{\pi}{2} + \frac{j\pi}{q}$, then $\sin(a) \geq 0$ and $\cos(a) < 0$. Therefore, since $y \leq 0$ by definition of $S_\alpha$, we get
\begin{equation}
\abs{f_m} \leq F_m e^{x\cos(a)} = C e^{-cx}
\end{equation}
for $C = F_m$ and $c = -\cos(a)$. Now if $a = \frac{3\pi}{2} - \frac{j\pi}{q}$, then
\begin{align}
x\cos(a) + y\sin(a) &= x\cos\left(\frac{3\pi}{2} - \frac{j\pi}{q}\right) + y\sin\left(\frac{3\pi}{2} - \frac{j\pi}{q}\right) \\
&= -x\sin\frac{j\pi}{q} - y\cos\frac{j\pi}{q} \\
&\leq -x\sin\frac{j\pi}{q} + x\cos\frac{j\pi}{q}\tan\frac{\pi}{2q}
\end{align}
where we used in the last line that $y \geq -x\tan \alpha$ by definition of $S_\alpha$. Since $\cos \frac{j\pi}{q} < \cos\frac{\pi}{2q}$, it follows that $\abs{f_m} \leq C e^{-cx}$ for $C = F_m$ and
\begin{equation}
c = \sin\frac{j\pi}{q} - \sin\frac{\pi}{2q} > 0.
\end{equation}
Combining our estimates on $f_m$ for $1 \leq m \leq 2q-2$, we can find positive constants $C$ and $c$ such that $\abs{v_\alpha^{\mathrm{d}}(x,y)} \leq Ce^{-cx}$. Now by computing explicitly the derivitives of $g_{a,b}$, one can show that
\begin{equation}
\abs{\del_x g_{a,b}} + \abs{\del_y g_{a,b}} \leq 4 \abs{g_{a,b}}
\end{equation}
and hence given our previous estimates on the functions $f_m$ in $v_\alpha^{\mathrm{d}}$, we can find positive constants $C$ and $c$ such that $\abs{\nabla_{(x,y)} v_\alpha^{\mathrm{d}}(x,y)} \leq Ce^{-cx}$. Combining both estimates on $v_\alpha^{\mathrm{d}}$ yields the result.
\end{proof}

\subsection{Surface wave quasimodes}

We can now use the surface wave solutions of the sloping beach problem to construct approximate solutions (quasimodes) for the sloshing problem on $\Sigma$. Let $\sigma$ be a real scaling factor. We consider the functions $v_\alpha(\sigma x)$ and $v_\beta(\sigma(L - x))$ corresponding to  solutions of the sloping beach problem starting off from the angles $\alpha$ and $\beta$ respectively. Let $v_{\alpha}^{\textup{p}}$ and $v_{\alpha}^{\textup{d}}$ correspond to the principal part and decaying parts of $v_\alpha$ on the boundary $I_1$ (as in Lemma \ref{lem:valpha}). In order for the sloping beach solutions to meet smoothly on $\mathcal{S}$, we want their principal parts to match. Therefore, we look for $\sigma$ such that
\begin{equation}\label{eq:quant}
v_\alpha^{\textup{p}}(\sigma x) = Q v_\beta^{\textup{p}}(\sigma(L-x)).
\end{equation}
for some non-zero $Q \in \C$. We call this the quantization condition. It fixes the values of $\sigma$ and leads to the quasimodes on $\Sigma$ given by
\begin{equation}\label{eq:quasisurface}
g_{\sigma}(x,y) = v_\alpha(\sigma x, \sigma y) + Q v_\beta^{\textup{d}}(\sigma (L-x), \sigma y) = Q v_\beta(\sigma (L-x), \sigma y) + v_\alpha^{\textup{d}}(\sigma x, \sigma y).
\end{equation}
Notice that $g_\sigma$ satisfies $\Delta g_\sigma = \mu^2 \sigma^2 g_\sigma = \lambda_n^2 g_\sigma$ in $\Sigma$ and $\del_\nu g_\sigma = \sigma g_{\sigma}$ on $\mathcal{S}$, but $\del_\nu g_\sigma \neq 0$ on $\mathcal{W}$ and hence it is not exactly a solution of \eqref{prob:varphi}. However, we have $\del_\nu v_\alpha = 0$ on the side making the angle $\alpha$ with $\mathcal{S}$, as well as $\del_\nu v_\beta = 0$ on the side making the angle $\beta$. The error term in $\del_\nu g_\sigma$ on each side of $\mathcal{W}$ therefore comes from the decaying part of the solution coming from the other side, which vanishes exponentially by Lemma \ref{lem:valpha}. Hence, the solution $g_\sigma$ is very close to being a solution of \eqref{prob:varphi}.

\section{Counting of quasi-eigenvalues}\label{sec:counting}

Let $N^e$ and $N^s$ denote the counting functions for the edge wave and surface wave quasi-eigenvalues respectively. The total counting function for quasi-eigenvalues then becomes $N^e + N^s$.

\subsection{Counting the edge wave quasi-eigenvalues}

Recall that for $0 \leq m < \frac{q-1}{2}$ the quasi-eigenvalue of the edge wave quasimode $\varphi^\alpha_{nm}(x,y)$ coming from the corner $\alpha$ is $\sigma^\alpha_{nm} = \frac{n\pi}{M} \sin(2m+1)\alpha$. Therefore, the eigenvalue counting function for one such quasimode is given by
\begin{equation}
\#\{n \in \N : \sigma^\alpha_{nm} < \sigma\} = \frac{M \sigma}{\pi \sin(2m+1)\alpha} + O(1).
\end{equation}
For $0 \leq \ell < \frac{r-1}{2}$, we have a similar expression for the eigenvalue counting function of each edge wave quasimode $\varphi^\beta_{n\ell}(L-x,y)$ coming from the corner $\beta$.

If $q$ and $r$ are both odd, we constructed another edge wave quasimode with eigenvalue $\frac{n\pi}{M}$. Hence, if we let $\nu_{\alpha,\beta} := qr \mod{2}$, the total eigenvalue couting function for the edge wave quasimodes is given by
\begin{equation}
N^e(\sigma) = \nu_{\alpha,\beta}\frac{M\sigma}{\pi} + \sum_{m=0}^{\floor{\frac{q}{2} - 1}} \frac{M\sigma}{\pi\sin(2m+1)\alpha} + \sum_{\ell=0}^{\floor{\frac{r}{2} - 1}} \frac{M\sigma}{\pi\sin(2\ell+1)\beta} + O(1)
\end{equation}
which is precisely the statement of Theorem \ref{thm:edge}.

An interesting thing to note is that the expression for $N^e$ only depends on the angles and $M$, the length of the side where the angles are on $\Omega$. It does not depend on $L$. This makes sense since the solutions mainly live along the side of length $M$ by Lemma \ref{lem:edge}.

\subsection{Finding the surface wave quasi-eigenvalues}

Suppose that $\alpha = \frac{\pi}{2q}$ and $\beta = \frac{\pi}{2r}$. By Lemma \ref{lem:valpha}, the principal part of $v_\alpha(\sigma x)$ is given by
\begin{equation}
v_\alpha^{\textup{p}}(\sigma x) = e^{-i\sqrt{1-\mu^2}\sigma x} + \gamma(\xi) e^{i\sqrt{1-\mu^2}\sigma x}
\end{equation}
where we can write $\gamma(\xi) = (-1)^{q-1} e^{2i\theta_\alpha}$ for
\begin{equation}\label{theta}
\theta_\alpha(n,\sigma) = \sum_{j=1}^{q-1} \arctan\left(\frac{\sqrt{1-\mu^2} \sin \frac{j\pi}{q}}{\cos \frac{j\pi}{q} - 1}\right) = -\sum_{j=1}^{q-1} \arctan\left(\frac{\sqrt{1-\left(\frac{n\pi}{\sigma M}\right)^2} \sin \frac{j\pi}{q}}{1 - \cos \frac{j\pi}{q}}\right).
\end{equation}
We have substituted $\mu = \frac{n\pi}{\sigma M}$ in the last equation. We have similar expressions for $v_\beta$. Since multiplying $v_\alpha$ and $v_\beta$ by constants still results in solutions of \eqref{beachsector}, we consider rather the functions $V_\alpha$ and $V_\beta$ where
\begin{equation}
V_\alpha(x) =
\begin{cases}
e^{-i\theta_\alpha} v_\alpha(x) & \text{if } q \text{ is odd,} \\
ie^{-i\theta_\alpha} v_\alpha(x) & \text{if } q \text{ is even,}
\end{cases}
\end{equation}
with $V_\beta$ defined similarly. Notice that if $q$ is odd, then the principal part of $V_\alpha$ is given by
\begin{equation}
V_\alpha^{\textup{p}}(x) = 2 \cos(\sqrt{1-\mu^2} x + \theta_\alpha)
\end{equation}
and if $q$ is even,
\begin{equation}
V_\alpha^{\textup{p}}(x) = 2 \sin(\sqrt{1-\mu^2} x + \theta_\alpha).
\end{equation}
The quantization condition \eqref{eq:quant} then becomes
\begin{equation}
V_\alpha(\sigma x) = \pm V_\beta(\sigma(L-x))
\end{equation}
which reduces to solving
\begin{equation}\label{eq:quantsigma}
\sqrt{1-\left(\frac{n\pi}{\sigma M}\right)^2} \sigma L = -(\theta_\alpha + \theta_\beta) + (m - \kappa_{\alpha,\beta})\pi
\end{equation}
for $m \in \Z$ and
\begin{equation}
\kappa_{\alpha,\beta} =
\begin{cases}
0 & \text{if } q \text{ and } r \text{ have the same parity,} \\
\frac{1}{2} & \text{otherwise.}
\end{cases}
\end{equation}
We can rewrite this equation as
\begin{equation}\label{eq:ellipsesigma}
\sigma^2 = \left(\frac{(m - \kappa_{\alpha,\beta} - \frac{1}{\pi}(\theta_\alpha + \theta_\beta))\pi}{L}\right)^2 + \left(\frac{n\pi}{M}\right)^2.
\end{equation}
It is important to keep in mind that $\theta_\alpha$ and $\theta_\beta$ depend on $\sigma$ and this is what makes the equation difficult to solve. In the case where $\alpha = \beta = \frac{\pi}{4}$, notice that equation \eqref{eq:ellipsesigma} coincides with the equation \eqref{eq:sigma} that we obtained from exact computation of the eigenfunctions. When $\kappa_{\alpha,\beta} = 0$, the trivial solution $m = 0$ and $\sigma = \frac{n\pi}{M}$ corresponds to the constant solution and we can ignore it. We wish to restrict ourselves to positive values of $m$ but since $-(\theta_\alpha + \theta_\beta) \geq 0$, we see that $m$ can take negative values in \eqref{eq:quantsigma}. However, there is only a finite number of such solutions.

\begin{lem}\label{lem:solutions}
There is at most a finite number of pairs $(m,n)$ with $m \leq 0 \leq n$ such that \eqref{eq:quantsigma} admits a nontrivial solution. Furthermore, for all $m > 0$ and $n \geq 0$, there exists a unique solution $\sigma_{m,n}$ of \eqref{eq:quantsigma}.
\end{lem}

\begin{proof}
First, we show the case $\alpha = \beta = \frac{\pi}{2q}$. For $n \in \N_0$, consider the functions $f_n : [\frac{n\pi}{M}, \infty) \rightarrow \R$ defined by
\begin{equation}
f_n(\sigma) = \frac{1}{\pi}\left(\sqrt{1 - \left(\frac{n\pi}{\sigma M}\right)^2} \sigma L + 2 \theta_\alpha(n,\sigma)\right).
\end{equation}
Notice that $f_n(\frac{n\pi}{M}) = 0$ and $f_{n}(\sigma)$ tends to infinity as $\sigma \rightarrow \infty$. Moreover, we can write
\begin{equation}\label{eq:derivee}
f_n'(\sigma) = \frac{1}{\pi\sqrt{1 - \left(\frac{n\pi}{\sigma M}\right)^2}}\left(L - 2\left(\frac{n\pi}{M}\right)^2 \sum_{j=1}^{q-1} \frac{\alpha_j}{\sigma^3[1 + \alpha_j^2 \left(1 - \left(\frac{n\pi}{\sigma M}\right)^2 \right)]}\right)
\end{equation}
with $\alpha_j = \frac{\sin \frac{j\pi}{q}}{1 - \cos \frac{j\pi}{q}} > 0$. When $\sigma$ increases, the value of the sum strictly decreases and tends to zero. Hence, even if $f_n'(\sigma) < 0$ for some values, it is eventually positive and tends to $\frac{L}{\pi}$ with the derivative vanishing at most once. When $n$ gets sufficiently large, so does $\sigma$, and the derivative is positive for all values of $\sigma$. In fact, when $\sigma \rightarrow \frac{n\pi}{M}^+$, the expression in parentheses in \eqref{eq:derivee} behaves like
\begin{equation}
L - \frac{2}{\sigma} \sum_{j=1}^{q-1} \alpha_j
\end{equation}
which is positive for $\sigma$ sufficiently large. Hence there exists $n_0$ such that $f_n'(\sigma) > 0$ for all $\sigma > \frac{n\pi}{M}$ and $n > n_0$.

We see that $\sigma$ is a solution of \eqref{eq:quantsigma} corresponding to given integers $m$ and $n$ if and only if $f_n(\sigma) = m$. From the previous calculations, there is only a finite number of $f_n$ which take negative values and the set $f_n^{-1}((-\infty,0])$ is bounded since $f_n$ tends to infinity as $\sigma \rightarrow \infty$. If $f_n$ takes negative values, it can then only take a finite number of negative integer values, and since its derivative vanishes exactly once, $f_n$ can be a given negative integer at most twice. Therefore, the set
\begin{equation}
\bigcup_{n = 0}^\infty f_n^{-1}(\Z_{<0})
\end{equation}
is finite and the first part of the lemma follows since we ignore the solutions with $\sigma = \frac{n\pi}{M}$ and $m = 0$. The second part of the lemma follows from the fact that $f_n'(\sigma) > 0$ whenever $f_n(\sigma) > 0$ and that $f_n$ tends to infinity.

The proof with $\alpha \neq \beta$ is similar. Indeed, we only need to change one $\theta_\alpha$ by $\theta_\beta + \kappa_{\alpha,\beta}$ in the definition of $f_n(\sigma)$. It is straightforward to see that $f_n'(\sigma)$ is eventually positive for all $n$ sufficiently big and since $f_n\left(\frac{n\pi}{M}\right) = \kappa_{\alpha, \beta}$, there is still a finite number of negative solutions.
\end{proof}

\subsection{Counting the surface wave quasi-eigenvalues}

Now that we know how to find the surface wave quasi-eigenvalues, we can count them in order to prove Theorem \ref{thm:surface}.

We know from Lemma \ref{lem:solutions} that there is only a finite number of solutions corresponding to non-positive values of $m$. They contribute $O(1)$ to the counting function and we can ignore them. Therefore, we restrict ourselves to solutions corresponding to $m > 0$ and $n \geq 0$. We also know that for each such pair $(m,n)$, there exists a unique solution of \eqref{eq:quantsigma}. We denote it by $\sigma_{m,n}$. Let $\sigma > 0$ and consider the set
\begin{equation}
E_\sigma = \left\{(x,y) \in \R^2 : \left(\frac{x \pi}{\sigma L}\right)^2 + \left(\frac{y \pi}{\sigma M}\right)^2 < 1\right\}.
\end{equation}
We have
\begin{equation}\label{eq:Esigma}
\#\{(m,n) \in E_\sigma \cap (\N \times \N_0)\} = \frac{LM}{4\pi} \sigma^2 + \frac{L-M}{2\pi} \sigma + o(\sigma)
\end{equation}
where the error term $o(\sigma)$ comes from known estimates on the Gauss circle problem (see \cite{Littlewood} for example). Suppose that $(m,n) \in E_\sigma$ and let $d > 0$ be the horizontal distance between $(m,n)$ and the boundary ellipse of $E_\sigma$, i.e.
\begin{equation}
d = x_n -m
\end{equation}
where $x_n$ is the positive solution to $\left(\frac{x_n \pi}{\sigma L}\right)^2 + \left(\frac{n \pi}{\sigma M}\right)^2 = 1$. From equation \eqref{eq:ellipsesigma}, we see that $\sigma_{m,n} < \sigma$ if and only if
\begin{equation}
m + f(n,\sigma_{m,n}) < x_n
\end{equation}
or equivalently $d > f(n,\sigma_{m,n})$ where
\begin{equation}
f(n,\sigma) = -\kappa_{\alpha,\beta} - \frac{1}{\pi}(\theta_\alpha(n,\sigma) + \theta_\beta(n,\sigma)).
\end{equation}
Notice that $f(n,\sigma)$ only depends on $\frac{n}{\sigma}$ and can hence be written as $f(\frac{n}{\sigma})$. We will use both notations. Therefore, counting the surface wave eigenvalues is equivalent (up to $O(1)$) to counting the total number of integer points $(m,n) \in E_\sigma$ with $m > 0$ and $n \geq 0$ to which we subtract the points such that $d \leq f(n,\sigma_{m,n})$. Denote by $N^s_{-}(\sigma)$ the number of such points, i.e.
\begin{equation}
N^s_{-}(\sigma) = \#\{(m,n) \in E_\sigma \cap (\N \times \N_0) : d \leq f(n,\sigma_{m,n})\}.
\end{equation}
From equation \eqref{eq:Esigma}, it then follows that
\begin{equation}
N^s(\sigma) = \frac{LM}{4\pi} \sigma^2 + \frac{L-M}{2\pi} \sigma - N^s_{-}(\sigma) + o(\sigma)
\end{equation}
and therefore proving Theorem \ref{thm:surface} is equivalent to proving the following.
\begin{thm}\label{thm:removed}
The counting function $N^s_-(\sigma)$ satisfies
\begin{equation}
N^s_-(\sigma) = \frac{M\sigma}{\pi} \int_0^{1} f(t) \de t + o(\sigma).
\end{equation}
\end{thm}

We start by giving an heuristic for this result. Let $\sigma_{m,n}$ be such that $(m,n) \in E_{\sigma}$ but $\sigma_{m,n} \geq \sigma$. We expect $\sigma_{m,n}$ to be relatively close to $\sigma$ in a way that $f(n,\sigma_{m,n})$ should be close to $f(n,\sigma)$. For simplicity of the argument, suppose that $f(n,\sigma_{m,n}) = f(n,\sigma)$. The boundary of the ellipse $E_\sigma$ in the first quadrant of the $(x,y)$ plane can be given by the curve
\begin{equation}
\tau_\sigma(t) = \left(\frac{\sigma L}{\pi}\sqrt{1 - \left(\frac{t\pi}{\sigma M}\right)^2}, t\right).
\end{equation}
for $t \in [0,\frac{M\sigma}{\pi}]$. Let $\gamma_\sigma : [0,\frac{\sigma M}{\pi}]$ be the curve
\begin{equation}
\gamma_\sigma(t) = \tau_\sigma(t) - (f(t,\sigma),0).
\end{equation}
Then, the integer points in $E_\sigma$ in the region bounded by $\gamma_\sigma$, $\tau_\sigma$ and the $x$-axis are precisely those such that $d \leq f(n,\sigma_{m,n})$, i.e. those that contribute to $N^s_-(\sigma)$. It is then reasonable to expect that the area of this region should be a good approximation for the number of integer points within it. The area is given by
\begin{equation}
\int_{0}^{\frac{M \sigma}{\pi}} f\left(\frac{t}{\sigma}\right) \de t = \frac{M \sigma}{\pi} \int_0^1 f(t) \de t.
\end{equation}
However, it could be that this approximation is not good at all since we took the area of a very thin strip which could miss all the integer points. For this estimate to be good, we need to show that the integer points are well-behaved, in the sense that they are evenly or uniformly distributed across this strip. To do so, we will rely on Weyl's equidistribution theorem.

In order to simplify the expressions, we now assume that $L = M = \pi$ and $\alpha = \beta$. However, the proofs will hold for all values. We will need the following two lemmas.

\begin{lem}\label{lem:approxsigma}
For all $m$ such that $x_n - q + 1 \leq m \leq x_n$ and all $0 \leq n \leq \sigma$, the estimate
\begin{equation}
f\left(\frac{n}{\sigma_{m,n}}\right) = f\left(\frac{n}{\sigma}\right) + o(1)
\end{equation}
holds uniformly in $m$ and $n$ as $\sigma \rightarrow \infty$.
\end{lem}

\begin{proof}
Since $0 \leq f(t) \leq q-1$ from equation \eqref{theta} for all $t \in [0,1]$, and
\begin{equation}
\sigma_{m,n} = \left(m + f\left(\frac{n}{\sigma_{m,n}}\right)\right)^2 + n^2
\end{equation}
it follows that
\begin{equation}
(x_n - q + 1)^2 + n^2 \leq \sigma_{m,n}^2 \leq (x_n + q - 1)^2 + n^2
\end{equation}
for $m$ satisfying $x_n - q + 1 \leq m \leq x_n$. Expanding each side and using the fact that $\sigma^2 = x_n^2 + n^2$ yields
\begin{equation}
\sigma^2 - 2(q-1)x_n + (q-1)^2 \leq \sigma_{m,n}^2 \leq \sigma^2 + 2(q-1)x_n + (q-1)^2.
\end{equation}
Since $\pm 2(q-1)x_n + (q-1)^2 = O(\sigma)$, it follows that $\sigma_{m,n}^2 = \sigma^2 + O(\sigma)$ and hence
\begin{equation}
\sigma_{m,n} = \sigma + O(1).
\end{equation}
Therefore,
\begin{equation}
\frac{n}{\sigma_{m,n}} = \frac{n}{\sigma + O(1)} = \frac{n}{\sigma} + O\left(\frac{n}{\sigma^2}\right).
\end{equation}
Since $0 \leq n \leq \sigma$, we get that
\begin{equation}
\frac{n}{\sigma_{m,n}} = \frac{n}{\sigma} + O\left(\frac{1}{\sigma}\right)
\end{equation}
uniformly in $n$ (and $m$). Since $f$ is uniformly continuous, it follows that, as $\sigma \rightarrow \infty$,
\begin{equation}
f\left(\frac{n}{\sigma_{m,n}}\right) = f\left(\frac{n}{\sigma} + O\left(\frac{1}{\sigma}\right)\right) = f\left(\frac{n}{\sigma}\right) + o(1).
\end{equation}
\end{proof}

\begin{lem}\label{lem:exp}
Fix $K \in \N$ and let $h \in \Z$ with $h \neq 0$. Let $e(x) = e^{2\pi i x}$. Then
\begin{equation}
\lim_{\sigma\rightarrow\infty} \frac{K}{\sigma} \sum_{\frac{r\sigma}{K} \leq n < \frac{(r+1)\sigma}{K}} e(h\sqrt{\sigma^2 - n^2}) = 0
\end{equation}
for all $0 \leq r \leq K-2$.
\end{lem}

To prove this lemma, we will need the following theorem from van der Corput \cite{Corput} on bounding exponential sums.

\begin{thm}[van der Corput \cite{Corput}, 1922]\label{thm:corput}
Let $F: I \rightarrow \R$ be a $C^2$ function on an interval $I$ with $\lambda \leq \abs{F''(x)} \leq \alpha \lambda$. Then
\begin{equation}
\sum_{n \in I} e(F(n)) \ll \alpha \abs{I} \lambda^{1/2} + \lambda^{-1/2}
\end{equation}
where the implied constant is absolute.
\end{thm}

\begin{proof}[Proof of Lemma \ref{lem:exp}]
The following proof is inspired by a proof provided to us by Zeev Rudnick. We apply Theorem \ref{thm:corput} with $I = \left[\frac{r\sigma}{K}, \frac{(r+1)\sigma}{K}\right)$ and $F(x) = h\sqrt{\sigma^2 - x^2}$. We have
\begin{equation}
F''(x) = - \frac{h\sigma^2}{(\sigma^2 - x^2)^{3/2}}.
\end{equation}
Since $\sigma^2 - x^2 \leq \sigma^2$, we have
\begin{equation}
\frac{\abs{h}}{\sigma} \leq \abs{F''(x)}.
\end{equation}
On the other hand, since $r \leq K-2$, we have $\sigma - x > \sigma - \frac{(r+1)\sigma}{K} \geq \frac{\sigma}{K}$ and hence
\begin{equation}
\abs{F''(x)} = \frac{\abs{h}\sigma^2}{((\sigma - x)(\sigma + x))^{3/2}} \leq \frac{\abs{h}\sigma^2}{(\frac{\sigma^2}{K})^{3/2}} = K^{3/2} \frac{\abs{h}}{\sigma}.
\end{equation}
Fixing $h$ and applying Theorem \ref{thm:corput} with $\lambda = \frac{\abs{h}}{\sigma}$ and $\alpha = K^{3/2}$ yields
\begin{equation}
\sum_{n \in I} e(F(n)) \ll_h K^{3/2} \frac{\sigma}{K} \frac{1}{\sqrt{\sigma}} + \sqrt{\sigma} = \sqrt{\sigma}(\sqrt{K} + 1)
\end{equation}
It follows that
\begin{equation}
\frac{K}{\sigma} \sum_{n \in I} e(h\sqrt{\sigma^2 - n^2}) \ll_h \frac{K^{3/2} + K}{\sqrt{\sigma}}
\end{equation}
which tends to $0$ as $\sigma \rightarrow \infty$.
\end{proof}

Denote by $d_n(\sigma)$ the distance between $x_n$ (the positive solution of $\sigma^2 = x_n^2 + n^2$) and the closest integer point $(m,n)$ satisfying $m^2 + n^2 < \sigma^2$. This distance is precisely the fractional part of $\sqrt{\sigma^2 - n^2}$. From Weyl's equidistribution theorem, Lemma \ref{lem:exp} is equivalent to the following lemma which will enable us to prove Theorem \ref{thm:removed}.

\begin{lem}\label{lem:equi}
Fix $K \in \N$. Then, for any interval $[\alpha,\beta] \subset [0,1]$ and for all $0 \leq r \leq K-2$,
\begin{equation}
\lim_{\sigma \rightarrow \infty} \frac{K}{\sigma} \#\left\{ n \in \left[\frac{r\sigma}{K}, \frac{(r+1)\sigma}{K}\right) : d_n(\sigma) \in [\alpha,\beta]\right\} = \beta - \alpha.
\end{equation}
\end{lem}

\begin{proof}[Proof of Theorem \ref{thm:removed}]
We wish to estimate
\begin{equation}
N^s_-(\sigma) = \sum_{(m,n) \in E_\sigma} \1{x_n - m \leq f(n,\sigma_{m,n})}
\end{equation}
since $\sigma_{m,n} < \sigma$ if and only if $x_n - m > f(n,\sigma_{m,n})$. Since $f$ is bounded by $q-1$, we have that
\begin{equation}
\1{x_n - m \leq f(n,\sigma_{m,n})} = 0
\end{equation}
for all $m$ such that $m < x_n - q + 1$. Hence,
\begin{equation}
N^s_-(\sigma) = \sum_{n = 0}^{\floor{\sigma}} \sum_{m = \ceil{x_n - q + 1}}^{\floor{x_n}} \1{x_n - m \leq f(n,\sigma_{m,n})}.
\end{equation}
From Lemma \ref{lem:approxsigma}, for the values of $n$ and $m$ present in the sum, we can find a function $h(\sigma)$ which goes to zero as $\sigma \rightarrow \infty$ such that
\begin{equation}\label{eq:oneapproxsigma}
\1{x_n - m \leq f(n,\sigma) - h(\sigma)} \leq \1{x_n - m \leq f(n,\sigma_{m,n})} \leq \1{x_n - m \leq f(n,\sigma) + h(\sigma)}.
\end{equation}
This motivates us to rather estimate the quantity
\begin{equation}
S(\sigma) := \sum_{n = 0}^{\floor{\sigma}} \sum_{m = \ceil{x_n - q + 1}}^{\floor{x_n}} \1{x_n - m \leq f(n,\sigma)}.
\end{equation}
Writing $m = \floor{x_n} - r$, this is equivalent to
\begin{equation}
\sum_{n = 0}^{\floor{\sigma}} \sum_{r=0}^{\floor{x_n} - \ceil{x_n - q + 1}} \1{x_n - \floor{x_n} \leq f(n,\sigma) - r}.
\end{equation}
Since $0 \leq x_n - \floor{x_n} < 1$, we see that
\begin{equation}
\1{x_n - \floor{x_n} \leq f(n,\sigma) - r} =
\begin{cases}
1 & \text{if } r \leq \floor{f(n,\sigma)} - 1 \\
\1{x_n - \floor{x_n} \leq f(n,\sigma) - r} & \text{if } r = \floor{f(n,\sigma)} \\
0 & \text{if } r > \floor{f(n,\sigma)} + 1.
\end{cases}
\end{equation}
Since $f$ is strictly decreasing, it takes integer values at most $q-1$ times. With a small error, we can therefore change the last condition to $r \geq \floor{f(n,\sigma)} + 1$. We then get
\begin{equation}
S(\sigma) = \left(\sum_{n=0}^{\floor{\sigma}} \floor{f(n,\sigma)} + \1{x_n - \floor{x_n} \leq f(n,\sigma) - \floor{f(n,\sigma)}}\right) + O(1)
\end{equation}
We now consider $\frac{S(\sigma)}{\sigma}$. We claim that
\begin{equation}
\lim_{\sigma \rightarrow \infty} \frac{S(\sigma)}{\sigma} = \int_0^1 f(t) \de t.
\end{equation}
Rewriting $f(n,\sigma)$ as $f\left(\frac{n}{\sigma}\right)$, the first term of $S(\sigma)$ yields
\begin{equation}
\lim_{\sigma\rightarrow \infty} \frac{1}{\sigma} \sum_{n=0}^{\floor{\sigma}} \floor{f\left(\frac{n}{\sigma}\right)} = \int_0^1 \floor{f(t)} \de t.
\end{equation}
Setting $g\left(\frac{n}{\sigma}\right) = f\left(\frac{n}{\sigma}\right) - \floor{f\left(\frac{n}{\sigma}\right)}$ and noticing that $x_n - \floor{x_n} = d_n(\sigma)$, it remains to estimate
\begin{equation}
\lim_{\sigma \rightarrow \infty} \frac{1}{\sigma} \sum_{n=0}^{\floor{\sigma}} \1{d_n(\sigma) \leq g\left(\frac{n}{\sigma}\right)}.
\end{equation}
Let $\varepsilon > 0$ and let $K \in \N$ be such that $\frac{1}{K} < \frac{\varepsilon}{3}$ and
\begin{equation}
\abs{\frac{1}{K} \sum_{r=0}^{K-1} g(x_r)- \int_{0}^1 g(t) \de t} < \frac{\varepsilon}{3}
\end{equation}
for all choices of $x_r \in \left[\frac{r}{K}, \frac{r+1}{K}\right]$. Such a $K$ exists since $g$ is piecewise continuous. Dividing $[0,\sigma]$ into $K$ subintervals, we get that
\begin{equation}
\frac{1}{\sigma} \sum_{n=0}^{\floor{\sigma}} \1{d_n(\sigma) \leq g\left(\frac{n}{\sigma}\right)} \leq \frac{1}{K} \sum_{r=0}^{K-1} \frac{K}{\sigma} \#\left\{n \in \left[\frac{r\sigma}{K}, \frac{(r+1)\sigma}{K}\right) : d_n(\sigma) \leq \sup_{x\in \left[\frac{r}{K}, \frac{r+1}{K}\right]} g(x)\right\}.
\end{equation}
The reverse inequality holds with the supremum replaced with the infimum. When $r = K-1$, we can use the trivial bound
\begin{equation}
\frac{K}{\sigma} \#\left\{n \in \left[\frac{r\sigma}{K}, \frac{(r+1)\sigma}{K}\right) : d_n(\sigma) \leq \sup_{x\in \left[\frac{r}{K}, \frac{r+1}{K}\right]} g(x)\right\} \leq 1.
\end{equation}
However, when $0 \leq r \leq K-2$, we can use Lemma \ref{lem:equi}. Together, this yields
\begin{equation}\label{eq:limsigma}
\lim_{\sigma \rightarrow \infty} \frac{1}{\sigma} \sum_{n=0}^{\floor{\sigma}} \1{d_n(\sigma) \leq g\left(\frac{n}{\sigma}\right)} \leq \frac{1}{K} \sum_{r=0}^{K-2} \sup_{x\in \left[\frac{r}{K}, \frac{r+1}{K}\right]} g(x) + \frac{1}{K} < \int_0^1 g(t) \de t + \varepsilon.
\end{equation}
Proceeding similarly with the reversed inequality, it follows that for all $\varepsilon > 0$,
\begin{equation}
\abs{\lim_{\sigma \rightarrow \infty} \frac{1}{\sigma} \sum_{n=0}^{\floor{\sigma}} \1{d_n(\sigma) \leq g\left(\frac{n}{\sigma}\right)} - \int_0^1 g(t) \de t} < \varepsilon
\end{equation}
and therefore
\begin{equation}
\lim_{\sigma \rightarrow \infty} \frac{S(\sigma)}{\sigma} = \int_0^1 \floor{f(t)} \de t + \int_0^1 f(t) - \floor{f(t)} \de t = \int_0^1 f(t) \de t.
\end{equation}

Finally, we see that if we were to change $f(n,\sigma)$ for $f(n,\sigma) \pm h(\sigma)$ with $h(\sigma)$ going to $0$ as $\sigma \rightarrow \infty$ in the definition of $S(\sigma)$, the result would still hold since \eqref{eq:limsigma} holds from the fact that for all $y \in [0,1]$,
\begin{equation}
\lim_{\sigma\rightarrow\infty} \frac{K}{\sigma} \#\left\{n \in \left[\frac{r\sigma}{K}, \frac{(r+1)\sigma}{K}\right) : d_n(\sigma) \leq y + o(1)\right\} = y.
\end{equation} From \eqref{eq:oneapproxsigma}, it then follows that
\begin{equation}
\lim_{\sigma \rightarrow \infty} \frac{N^s_-(\sigma)}{\sigma} = \int_0^1 f(t) \de t
\end{equation}
so that
\begin{equation}
N^s_-(\sigma) = \sigma \int_0^1 f(t) \de t + o(\sigma).
\end{equation}
\end{proof}

\section{Quasimode analysis and numerical evidence}\label{sec:discuss}

The results we have presented are only approximate solutions of problem \eqref{sloshing}. However, we will show that there is an actual eigenvalue of the problem near every quasi-eigenvalue and our numerical experiments seem to agree with both our conjectures.

\subsection{Analysis of the quasi-eigenvalues}

For $n \in \N_0$, let $\{\tilde{\sigma}_j^{(n)}\}_{j \in \N}$ denote the set of quasi-eigenvalues (coming from both our edge waves and surface waves solutions) indexed in ascending order for which the quasimodes solve $\Delta \varphi = \lambda_n^2 \varphi$ in $\Sigma$, and let $\{\sigma_k^{(n)}\}_{k \in \N}$ denote the set of real eigenvalues (sloshing eigenvalues) of problem \eqref{prob:varphi}. The following lemma is analogous to Lemma $2.6$ in \cite{LPPS1}.

\begin{lem}\label{lem:LPPS}
There exist positive constants $C$ and $c$ such that for every $n \in \N_0$ and $j \in \N$, there exists $k \in \N$ such that
\begin{equation}\label{eq:LPPS}
\abs{\tilde{\sigma}_j^{(n)} - \sigma_k^{(n)}} \leq C e^{-c\tilde{\sigma}_j^{(n)}}.
\end{equation}
\end{lem}

In order to prove it, we need a preliminary result on our quasimodes. We denote by $\varphi_\sigma$ a quasimode with quasi-eigenvalue $\sigma$.

\begin{prop}\label{prop:Phisigma}
There exist positive constants $C$ and $c$ such that for any quasimode $\varphi_\sigma$,
\begin{equation}
\abs{\del_\nu \varphi_\sigma} \leq C e^{-c\sigma}
\end{equation}
for all $(x,y) \in \mathcal{W}$.
\end{prop}

\begin{proof}
Let us denote by $\mathcal{W}_\alpha$ and $\mathcal{W}_\beta$ the segments of $\mathcal{W}$ making angles $\alpha$ and $\beta$ with $\mathcal{S}$ respectively. We will again abuse notation when using $C$ and $c$ and we will use the fact that $C_1 \sigma e^{-c_1\sigma x} \leq C_2 e^{-c_2 \sigma}$ whenever $x$ is bounded from below by a positive number.

Firstly, if $\varphi_\sigma$ is an edge wave quasimode of the form $\varphi_{nm}^\alpha$ with $m \neq \frac{q-1}{2}$, then $\del_\nu \varphi_\sigma = 0$ on $\mathcal{W}_\alpha$. Moreover, by Lemma \ref{lem:edge}, since $\sigma \geq \lambda_n\sin \frac{\pi}{2q}$, we can find $C, c > 0$ such that
\begin{equation}
\abs{\del_\nu \varphi_\sigma} \leq C e^{-c\sigma}
\end{equation}
on $\mathcal{W}_\beta$. The same reasoning applies if $\varphi_\sigma$ is an edge wave quasimode of the form $\varphi_{n\ell}^\beta$ with $\ell \neq \frac{r-1}{2}$.

Secondly, if $\varphi_\sigma$ is the edge wave quasimode given by $\psi_n$ as in \eqref{eq:edgepsi}, then on $\mathcal{W}_\beta$, we have
\begin{equation}
\abs{\del_\nu \psi_n} = \abs{\del_\nu (A^\alpha_{\ell\ell} \varphi^\alpha_{nm}(x,y) - A^\alpha_{mm} A_{\ell\ell}^\beta e^{\lambda_n y})}
\end{equation}
since $(\del_\nu \varphi_{n\ell}^\beta)|_{\mathcal{W}_\beta} = 0$. Applying the estimate \eqref{eq:estimate4} from Lemma \ref{lem:edge} to that last expression yields $\abs{\del_\nu \psi_n} \leq C e^{-c\sigma}$ on $\mathcal{W}_\beta$. A similar reasoning yields the same estimate on $\mathcal{W}_\alpha$, and therefore on all $\mathcal{W}$.

Finally, if $\varphi_\sigma = g_\sigma$ is a surface wave quasimode given by equation \eqref{eq:quasisurface}, then by using the second expression for $g_\sigma$, we see that we have on $\mathcal{W}_\beta$
\begin{equation}
\abs{\del_\nu g_\sigma} = \abs{\del_\nu (v_\alpha^{\mathrm{d}} (\sigma x, \sigma y))}
\end{equation}
since $(\del_\nu v_\beta)|_{\mathcal{W}_\beta} = 0$. The estimate on the gradient of $v_\alpha^{\mathrm{d}}$ in Lemma \ref{lem:valpha} gives us our desired bound on $\mathcal{W}_\beta$. By using the first expression for $g_\sigma$, we can do the same reasoning on $\mathcal{W}_\alpha$, showing that $\abs{\del_\nu g_\sigma} \leq C e^{-c\sigma}$ everywhere on $\mathcal{W}$.

In all our calculations, both $C$ and $c$ depend solely on the angles $\alpha$ and $\beta$. The claim then follows.
\end{proof}

\begin{proof}[Proof of Lemma \ref{lem:LPPS}]
We will follow the argument laid out in Section $2.2$ of \cite{LPPS1} and slightly adapt it to our case. We refer to \cite{LPPS1} for further details of the argument.

Given one of our quasimodes $\varphi_\sigma$ satisfying $\Delta\varphi_\sigma = \lambda_n^2 \varphi_\sigma$ in $\Sigma$ and $\del_\nu \varphi_\sigma = \sigma \varphi_\sigma$ on $\mathcal{S}$, consider a function $\eta_\sigma$ that is solution of
\begin{equation}\label{prob:etasigma}
\begin{cases}
\Delta \eta_\sigma = \lambda_n^2 \eta_\sigma & \text{in } \Sigma, \\
\del_\nu \eta_\sigma = \del_\nu \varphi_\sigma & \text{on } \mathcal{W}, \\
\del_\nu \eta_\sigma = -\left(\int_{\mathcal{W}} \del_\nu \varphi_\sigma\right) \psi & \text{on } \mathcal{S},
\end{cases}
\end{equation}
where $\psi \in C^\infty(\mathcal{S})$ is a fixed function supported away from the the corners $\alpha$ and $\beta$ with $\int_{\mathcal{S}} \psi = 1$. The function $\eta_\sigma$ is the result of the Neumann-to-Dirichlet map $\ND_{-\lambda_n^2} : L^2(\del\Sigma) \rightarrow L^2(\del\Sigma)$ when applied to the function
\begin{equation}
h_\sigma =
\begin{cases}
\del_\nu \varphi_\sigma & \text{on } \mathcal{W}, \\
-\left(\int_{\mathcal{W}} \del_\nu \varphi_\sigma\right)\psi & \text{on } \mathcal{S}.
\end{cases}
\end{equation}
When $n = 0$, as mentioned in \cite{LPPS1}, such a solution $\eta_\sigma$ exists up to a constant and is therefore unique if we demand that $\int_{\del\Sigma} \eta_\sigma = 0$. Moreover, when acting on functions with mean-value $0$ on $\mathcal{S}$, $\ND_0$ is bounded. Now if $n > 0$, the operator $\ND_{-\lambda_n^2}$ is well-defined since $-\lambda_n^2 < 0$ is not a Neumann eigenvalue of $-\Delta$ on $\Sigma$ and it is a self-adjoint compact operator on $L^2(\del\Sigma)$ \cite{behrndt}. Moreover, the operators $\ND_{-\lambda_n^2}$ are uniformly bounded on $L^2(\del\Sigma)$ since their eigenvalues decrease when $n$ increases. This is due to the fact that $\ND_{-\lambda}$ is the inverse of the Dirichlet-to-Neumann map $\DN_{-\lambda}$ whose eigenvalues are positive and strictly increasing for $\lambda$ in the interval $(\varepsilon,\infty)$, see \cite{arendt} or \cite{friedlander}. It follows from Proposition \ref{prop:Phisigma} that
\begin{equation}\label{eq:etasigma}
\norm{\eta_\sigma}_{L^2(\mathcal{S})} \leq \norm{\ND_{-\lambda_n^2} h_\sigma}_{L^2(\del\Sigma)} \leq C \norm{h_\sigma}_{L^2(\del\Sigma)} \leq Ce^{-c\sigma}
\end{equation}
where the constants do not depend on $n$ nor $\sigma$.

The function $v_\sigma := \varphi_\sigma - \eta_\sigma$ satisfies $\Delta v_\sigma = \lambda_n^2 v_\sigma$ and its normal derivative vanishes on $\mathcal{W}$. Let $\DN_{-\lambda_n^2}$ now denote the Dirichlet-to-Neumann map that takes $f \in L^2(\mathcal{S})$ and maps it to $(\del_\nu \tilde{f})|_{\mathcal{S}}$ where $\Delta \tilde{f} = \lambda_n^2 \tilde{f}$ in $\Sigma$, $\del_\nu \tilde{f} = 0$ on $\mathcal{W}$, and $\tilde{f} = f$ on $\mathcal{S}$. Then, by construction, we have
\begin{equation}
\DN_{-\lambda_n^2} (v_\sigma|_{\mathcal{S}}) = (\del_\nu v_\sigma)|_{\mathcal{S}}.
\end{equation}
Since $\del_\nu \varphi_\sigma = \sigma \varphi_\sigma$ on $\mathcal{S}$, for every quasi-eigenvalue $\sigma$ we have
\begin{equation}\label{eq:key}
\norm{\DN_{-\lambda_n^2} (v_\sigma |_{\mathcal{S}}) - \sigma v_\sigma}_{L^2(\mathcal{S})} = \norm{\del_\nu \eta_\sigma - \sigma\eta_\sigma}_{L^2(\mathcal{S})} \leq Ce^{-c\sigma}
\end{equation}
where the last inequality follows from \eqref{eq:etasigma} and Proposition \ref{prop:Phisigma}.

By rescaling, suppose now that $\norm{v_\sigma}_{L^2(\Gamma_S)} = 1$ and let $(\phi_k^{(n)})_{k \geq 0}$ be a complete set of orthonormal eigenfunctions of $\DN_{-\lambda_n^2}$ with eigenvalues $\sigma_k^{(n)}$. Then, we can find coefficients $a_k = (v_\sigma, \phi_k)$ such that $\sum_{k=0}^\infty a_k^2 = 1$ and
\begin{equation}
v_\sigma = \sum_{k=0}^\infty a_k \phi_k.
\end{equation}
It follows from \eqref{eq:key} that
\begin{equation}
\norm{\DN_{-\lambda_n^2}(v_\sigma |_{\mathcal{S}}) - \sigma v_\sigma}_{L^2(\mathcal{S})}^2 = \sum_{k=0}^\infty a_k^2(\sigma_k^{(n)} - \sigma)^2 \leq C e^{-2c\sigma}
\end{equation}
and since $\sum_{k=0}^\infty a_k^2 = 1$, there must be a $k$ such that $(\sigma_k^{(n)} - \sigma)^2 \leq C e^{-2c\sigma}$ and therefore
\begin{equation}\label{eq:sigmasigmak}
\abs{\sigma - \sigma_k^{(n)}} \leq C e^{-c\sigma}.
\end{equation}
Plugging $\sigma = \tilde{\sigma}^{(n)}_j$ into \eqref{eq:sigmasigmak} yields \eqref{eq:LPPS}.
\end{proof}

We now have all the tools to prove Theorem \ref{thm:gg}.

\begin{proof}[Proof of Theorem \ref{thm:gg}]
We start by showing $N(\sigma) \geq N^s(\sigma) + o(\sigma)$. In order to get this estimate, we need to show that every surface wave quasi-eigenvalue is sufficiently isolated in order for every actual eigenvalue given by Lemma \ref{lem:LPPS} to be distinct. Denote the set of surface wave quasi-eigenvalues that solve \eqref{prob:varphi} for a given $n$ by $\{\overline{\sigma}^{(n)}_j\}_{j\in\N}$. First of all, given $n \neq n'$, we know that the real eigenvalues corresponding to $\overline{\sigma}^{(n)}_j$ and $\overline{\sigma}^{(n')}_{j'}$ are distinct eigenvalues of problem \eqref{sloshing} for all $j, j' \in \N$, since the corresponding eigenfunctions solve the equation $\Delta u = \lambda u$ in $\Sigma$ for different values of $\lambda$. By distinct, we do not necessarily mean that the eigenvalues are not equal, but rather that they correspond to different linearly independent eigenfunctions.

Recall that $\sigma > \frac{n\pi}{M}$ is a quasi-eigenvalue of a surface wave $\varphi_\sigma$ satisfying $\Delta \varphi_\sigma = \lambda_n^2\sigma$ if and only if
\begin{equation}
f_n(\sigma) = \frac{1}{\pi}\left(\sqrt{1 - \left(\frac{n\pi}{\sigma M}\right)^2} \sigma L + \theta_\alpha(n,\sigma) + \theta_\beta(n,\sigma)\right) + \kappa_{\alpha,\beta}
\end{equation}
is an integer (see Lemma \ref{lem:solutions} and its proof). Moreover, there exists $n_0 \in \N$ such that for all $n \geq n_0$ the function $f_n: \left[\frac{n\pi}{M}, \infty\right) \rightarrow \R$ is always positive and its derivative strictly decreases and tends to $\frac{L}{\pi}$. Therefore, for $n \geq n_0$, the eigenvalues $\overline{\sigma}^{(n)}_j$ satisfy
\begin{equation}
f_n(\overline{\sigma}^{(n)}_j) = j.
\end{equation}
By convexity of $f_n$, it follows that
\begin{equation}\label{eq:sigmaconsecutif}
\abs{\overline{\sigma}^{(n)}_{j+1} - \overline{\sigma}^{(n)}_j} \geq \overline{\sigma}^{(n)}_1 - \frac{n\pi}{M}
\end{equation}
where $f_n(\overline{\sigma}^{(n)}_1) = 1$. Since $\theta_\alpha$ and $\theta_\beta$ are both negative, we have
\begin{equation}
f_n(x) \leq h_n(x) := \frac{1}{\pi}\sqrt{1 - \left(\frac{n\pi}{x M}\right)^2} x L + \kappa_{\alpha,\beta}
\end{equation}
for all $x \geq \frac{n\pi}{M}$. Letting
\begin{equation}
x_1 = \sqrt{\left(\frac{\pi(1 - \kappa_{\alpha,\beta})}{L}\right)^2 + \left(\frac{n\pi}{M}\right)^2}
\end{equation}
we see that $h_n(x_1) = 1$. Since $h_n$ is strictly increasing, it follows that $f_n(x) < 1$ for all $x < x_1$ and therefore $\overline{\sigma}_1^{(n)} \geq x_1$. Consequently,
\begin{equation}
\overline{\sigma}_1^{(n)} - \frac{n\pi}{M} \geq  x_1 - \frac{n\pi}{M} \geq \frac{C}{n}
\end{equation}
where $C$ can be chosen independently of $n$. From \eqref{eq:sigmaconsecutif}, we get
\begin{equation}
\abs{\overline{\sigma}^{(n)}_{j+1} - \overline{\sigma}^{(n)}_j} \geq \frac{C}{n}.
\end{equation}
Hence, using that $\overline{\sigma}^{(n)}_j > \frac{n\pi}{M}$, Lemma \ref{lem:LPPS} guarantees that given $n$ sufficiently large the real eigenvalue next to $\overline{\sigma}^{(n)}_j$ is distinct for each $j\in\N$.

Now suppose that $n$ isn't large enough for the previous approach to apply. We know that $f_n'$ tends to $\frac{L}{\pi}$ and so there exists $j_0^{(n)} \in \R$ such that
\begin{equation}
\overline{\sigma}_j^{(n)} = \frac{\pi}{L} (j - j_0^{(n)}) + o_n(j).
\end{equation}
Therefore, there exists a constant $C_n$ such that for all $j$ sufficiently large
\begin{equation}
\abs{\overline{\sigma}_{j+1}^{(n)} - \overline{\sigma}_j^{(n)}} \geq C_n.
\end{equation}
Since $\overline{\sigma}_j^{(n)} \geq Cj$, Lemma \ref{lem:LPPS} then guarantees that if $j$ and $j'$ are sufficiently large, the sloshing eigenvalues next to $\overline{\sigma}_j^{(n)}$ and $\overline{\sigma}_{j'}^{(n)}$ are distinct as long as $j \neq j'$.

In short, all the sloshing eigenvalues $\sigma_k^{(n)}$ given by Lemma \ref{lem:LPPS} close to the surface wave quasi-eigenvalues $\overline{\sigma}_j^{(n)}$ are distinct as long as either $n$ or $j$ is sufficiently large. Thus, only a finite number of such sloshing eigenvalues can be identical. Denote that number by $P$. Then, we have
\begin{equation}
N^s(\sigma - Ce^{-c\sigma}) - P \leq N(\sigma)
\end{equation}
for all $\sigma \geq 0$. Our knowledge of $N^s(\sigma)$ guarantees that $N^s(\sigma - Ce^{-c\sigma}) = N^s(\sigma) + o(\sigma)$, which yields $N(\sigma) \geq N^s(\sigma) + o(\sigma)$.

Let us now consider the edge wave quasimodes. As in the case of the surface wave quasi-eigenvalues, the sloshing eigenvalues given by Lemma \ref{lem:LPPS} for different values of $n$ have to be distinct since the underlying eigenfunctions solve different equations inside $\Sigma$.

We consider first the quasimodes $\varphi_{nm}^{\alpha}$ and $\varphi_{n\ell}^{\beta}$ for $n \in \N$, $0 \leq m < \frac{q-1}{2}$ and $0 \leq \ell < \frac{r-1}{2}$, with quasi-eigenvalues given by
\begin{equation}
\sigma_{nm}^{\alpha} = \lambda_n \sin(2m+1) \alpha
\end{equation}
and
\begin{equation}
\sigma_{n\ell}^{\beta} = \lambda_n \sin(2\ell+1) \beta.
\end{equation}
If there are values of $m$ and $\ell$ such that $(2m+1)r = (2\ell+1)q$, then some quasi-eigenvalues $\sigma_{nm}^{\alpha}$ and $\sigma_{n\ell}^{\beta}$ have multiplicity $2$ and we will deal with them afterwards. Suppose for now that there are no such values of $m$ and $\ell$. Then, there exists $\delta > 0$ such that, given $n$, every edge wave quasi-eigenvalue is spaced by $\delta$ and at distance at least $\delta$ from $\frac{n\pi}{M}$. Lemma \ref{lem:LPPS} then guarantees that, except for maybe a finite number of them, all the real eigenvalues associated to those edge wave quasi-eigenvalues are distinct, and distinct from the ones we recovered close to the surface wave quasi-eigenvalues.

If $q$ and $r$ are both odd, we also have to consider the quasimodes $\psi_n$ with eigenvalue $\lambda_n = \frac{n\pi}{M}$. Since $\overline{\sigma}_1^{(n)} - \frac{n\pi}{M} \geq \frac{C}{n}$ and each other edge wave quasi-eigenvalue $\sigma_{nm}^\alpha$ or $\sigma_{nm}^\beta$ is at a distance at least $\delta$ from $\frac{n\pi}{M}$, it follows from Lemma \ref{lem:LPPS} that, except for maybe a finite number of them, all the real eigenvalues close to a quasi-eigenvalue $\lambda_n$ are distinct from the ones we found previously.

Suppose now that there exist $m_0 < \frac{q-1}{2}$ and $\ell_0 < \frac{r-1}{2}$ such that $(2m_0 + 1)r = (2\ell_0 + 1)q$. In other words, suppose that there are edge wave quasi-eigenvalues with multiplicity 2 since $\sigma_{nm_0}^\alpha = \sigma_{n\ell_0}^\beta$ for all $n \in \N$. Let us show that the multiplicity guarantees the presence of two distinct sloshing eigenvalues. Fix $n \in \N$ and let $\varphi_{\alpha}$, $\varphi_{\beta}$ and $\sigma$ denote respectively $\varphi_{nm_0}^\alpha$, $\varphi_{n\ell_0}^\beta$ and $\sigma_{nm_0}^\alpha$. Now let 
\begin{equation}
v_\alpha= \varphi_{\alpha} - \eta_{\alpha}
\end{equation}
where $\eta_\alpha$ is the solution of \eqref{prob:etasigma} for $\varphi_\sigma = \varphi_{\alpha}$. Rescaling if need be, suppose further that $v_\alpha$ has unit norm in $L^2(S)$. Then, by \eqref{eq:key} and Theorem 4.1 in \cite{LPPS1}, we can find a function $w_\alpha$ such that
\begin{itemize}
\item $w_\alpha$ is a linear combination of eigenfunctions of $\DN_{-\lambda_n^2}$ with eigenvalues in the interval $[\sigma - \sqrt{C}e^{-c\sigma/2}, \sigma + \sqrt{C}e^{-c\sigma/2}]$;
\item $\norm{w_\alpha}_{L^2(\mathcal{S})} = 1$;
\item $\norm{v_\alpha - w_\alpha}_{L^2(\mathcal{S})} \leq 2\sqrt{C}e^{-c\sigma/2}(1 + o_\sigma(1))$.
\end{itemize}
Here, $C$ and $c$ are the same constants as in Lemma \ref{lem:LPPS}. Divide the boundary $\mathcal{S}$ into two parts $\mathcal{S_\alpha} = [0,L/2] \times \{0\}$ and $\mathcal{S_\beta} = (L/2,L] \times \{0\}$. Then, we have
\begin{equation}
\norm{w_\alpha}_{L^2(\mathcal{S}_\beta)} \leq \norm{\varphi_{\alpha}}_{L^2(\mathcal{S}_\beta)} + \norm{\eta_\alpha}_{L^2(\mathcal{S})} + \norm{w_\alpha - v_\alpha}_{L^2(\mathcal{S})}.
\end{equation}
By Lemma \ref{lem:edge}, equation \eqref{eq:key} and the definition of $w_\alpha$, each of the terms on the right-hand side of the last equation vanish exponentially fast as $\sigma$ (and therefore $n$) goes to infinity. It follows that $\norm{w_\alpha}_{L^2(\mathcal{S}_\beta)}$ goes to $0$ as $n \rightarrow \infty$. We can repeat all of the previous construction for the angle $\beta$ to get a function $w_\beta$ with the same properties as $w_\alpha$ but with respect to $v_\beta = \varphi_\beta - \eta_\beta$. By the same arguments as above, $\norm{w_\beta}_{L^2(\mathcal{S}_\alpha)}$ goes to $0$ as $n \rightarrow \infty$ and therefore
\begin{equation}
\norm{w_\beta}_{L^2(\mathcal{S_\beta})} = \norm{w_\beta}_{L^2(\mathcal{S})} - \norm{w_\beta}_{L^2(\mathcal{S_\alpha})}
\end{equation}
goes to $1$ since $\norm{w_\beta}_{L^2(\mathcal{S})} = 1$. Both $w_\alpha$ and $w_\beta$ have unit norm in $L^2(\mathcal{S})$, but $\norm{w_\alpha}_{L^2(\mathcal{S}_\beta)} \rightarrow 0$ while $\norm{w_\beta}_{L^2(\mathcal{S}_\beta)} \rightarrow 1$. Thus, for $n$ sufficiently large, the two functions must be linearly independent. It follows that there are at least two eigenfunctions of $\DN_{-\lambda_n^2}$ with eigenvalues in the interval $[\sigma - \sqrt{C}e^{-c\sigma/2}, \sigma + \sqrt{C}e^{-c\sigma/2}]$. For $n$ sufficiently large, those eigenvalues must be distinct from all the previous sloshing eigenvalues that we found previously. Therefore, there are indeed 2 distinct sloshing eigenvalues close to each edge wave quasi-eigenvalue of multiplicity 2 that is sufficiently large.

Since the sloshing eigenvalues from Lemma \ref{lem:LPPS} that are close to the edge wave and surface wave quasimodes are distinct (except for maybe a finite number of them), we can combine them using the same trick we used for comparing $N^s$ and $N$. This yields
\begin{equation}
N(\sigma) \geq N^s(\sigma) + N^e(\sigma) + o(\sigma)
\end{equation}
as claimed.
\end{proof}

\subsection{Discussion on quasimodes} \label{sec:quasidisc}

We have shown that the counting function of our quasimodes bounds the real eigenvalue counting function from below, but in order to prove Conjecture \ref{conjecture}, we also need to prove that it bounds it from above. This should require showing that our quasi-eigenvalues approximate all the sloshing eigenvalues, which should be much more difficult to prove and require new ideas. In dimension $2$, it turns out that the quasimodes solve a Sturm-Liouville equation on the sloshing part of the boundary. This fact was used in \cite{LPPS1} to show that their quasimodes formed a complete set, and hence approximated every eigenfunction. Their method could work in our case, but we were unable to find an analogous Sturm-Liouville equation solved by our quasimodes. Furthermore, the presence of edge waves makes it even more complicated.

We only considered the cases where the angles $\alpha$ and $\beta$ were of the form $\frac{\pi}{2q}$. Note that our construction of the edge wave quasimodes is valid for any angle smaller than $\frac{\pi}{2}$. However, we used the fact that the angles were of the form $\frac{\pi}{2q}$ to construct explicitly the surface wave solutions of the sloping beach problem that we used in our quasimodes. Indeed, if we were to repeat the steps in Section \ref{sec:surfacesloping} for an arbitrary angle which is a rational multiple of $\pi$, the iterations of the operators $\mathcal{A}$ and $\mathcal{B}$ would lead to solutions that blow up at infinity and an analogous version of Lemma \ref{lem:valpha} would not hold. There might be a way to remedy this, but we were unable to do so. Moreover, we are unsure how the counting function behaves for arbitrary angles. In two dimensions, solutions due to Peters \cite{peters2d} allow to create quasimodes for arbitrary angles. Using the ideas of Peters in \cite{peters3d}, it should be possible to find similar solutions in three dimensions, which could lead to finding an expression of $N(\sigma)$ for arbitrary angles.

\subsection{Numerical evidence supporting Conjectures \ref{conjecture} and \ref{conj:complete}} \label{sec:num}

We now present numerical evidence to support both our conjectures. Let $\Sigma$ be the triangle of angles $\alpha$ and $\beta$ with sidelength $L$ resulting from the separation of variable on $\Omega$ (as in Figure \ref{fig:domain}). We used FreeFem++ to solve problem \eqref{prob:varphi} using the finite element method. It is a 2-d problem and hence much faster to solve than its 3-d counterpart of solving directly problem \eqref{sloshing} on all $\Omega$.

For simplicity, we take $L = M = \pi$. We start by computing $N(\sigma)$ up to $\sigma = 50$ for all the combinations of $\alpha$ and $\beta$ in the set $\{\frac{\pi}{4}, \frac{\pi}{6}, \frac{\pi}{8}\}$. In order to do so, we compute the first eigenvalues corresponding to $\lambda_n = n$ for sufficiently many $n$'s. We order and denote those eigenvalues by $\sigma_k(n)$. Note that from a theorem by Friedlander \cite{friedlander}, the eigenvalue $\sigma_k(n)$ gets larger as $n$ increases. Therefore, we only need to compute these eigenvalues until $\sigma_1(n) > 50$ and we can reduce the number of computed eigenvalues at each step in order to speed up the computations.

Consider the function
\begin{equation}
S(\sigma) := \frac{1}{\sigma} \left(N(\sigma) - \frac{LM}{4\pi} \sigma^2\right) = \frac{1}{\sigma} \left(N(\sigma) - \frac{\pi}{4} \sigma^2 \right).
\end{equation}
Then, Conjecture \ref{conjecture} is equivalent to showing
\begin{equation}
\lim_{\sigma \rightarrow \infty} S(\sigma) = N^s(1) + N^e(1) - \frac{\pi}{4}
\end{equation}
where here $N^s$ and $N^e$ are the expressions from Theorems \ref{thm:surface} and \ref{thm:edge} without the error terms. The plots in Figure \ref{fig:Ssigma} show our estimated value of $S(\sigma)$ for $0 \leq \sigma \leq 50$, as well as the value of $N^e(1) + N^s(1) - \frac{\pi}{4}$ to which it should converge when $\sigma$ tends to infinity.

When computing the eigenvalues numerically, we found that our quasi-eigenvalues matched them quite accurately. We have an exact expression for the edge wave quasi-eigenvalues from equation \eqref{eq:edgevalues} and we can compute the surface wave quasi-eigenvalues by solving equation \eqref{eq:quantsigma} for different values of $m$ (without forgetting that $m$ can take negatives values if $n$ is small). We did so using the function FindRoot in Mathematica. Tables \ref{table:pi4pi6}, \ref{table:pi10pi10}, \ref{table:pi6pi18} show the first quasi-eigenvalues computed with Mathematica as well as the first sloshing eigenvalues computed with FreeFEM++ for different values of $\alpha$ and $\beta$. As we conjectured, our quasi-eigenvalues seem to be asymptotically complete since they match the sloshing eigenvalues starting from a certain index. We have shifted the tables to highlight their matching. 

\newpage

\begin{figure}[htbp]
\captionsetup[subfigure]{labelformat=empty}
\centering
\begin{minipage}{0.5\textwidth}
  \centering
\includegraphics[width=0.9\textwidth]{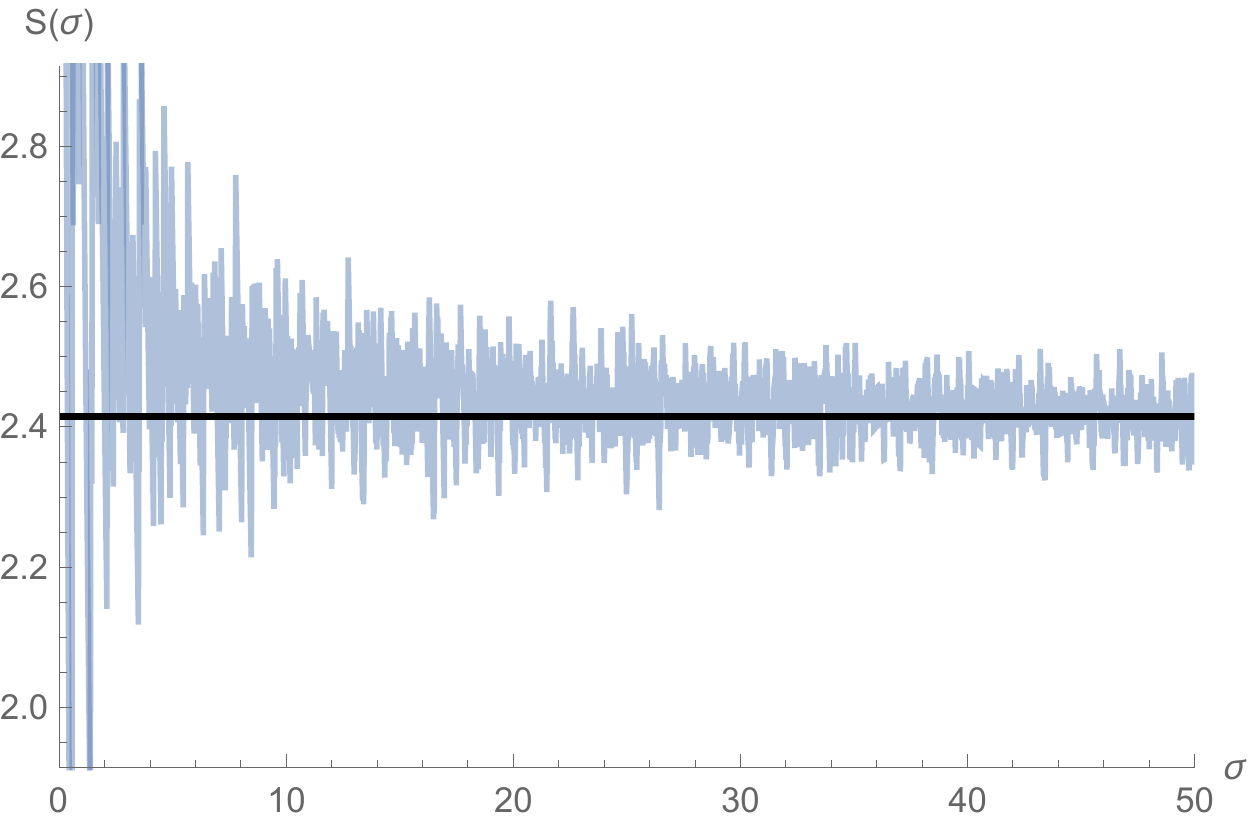}
\subcaption[first caption.]{$\alpha = \frac{\pi}{4}, \beta = \frac{\pi}{4}$}
\end{minipage}%
\begin{minipage}{0.5\textwidth}
  \centering
\includegraphics[width=0.9\textwidth]{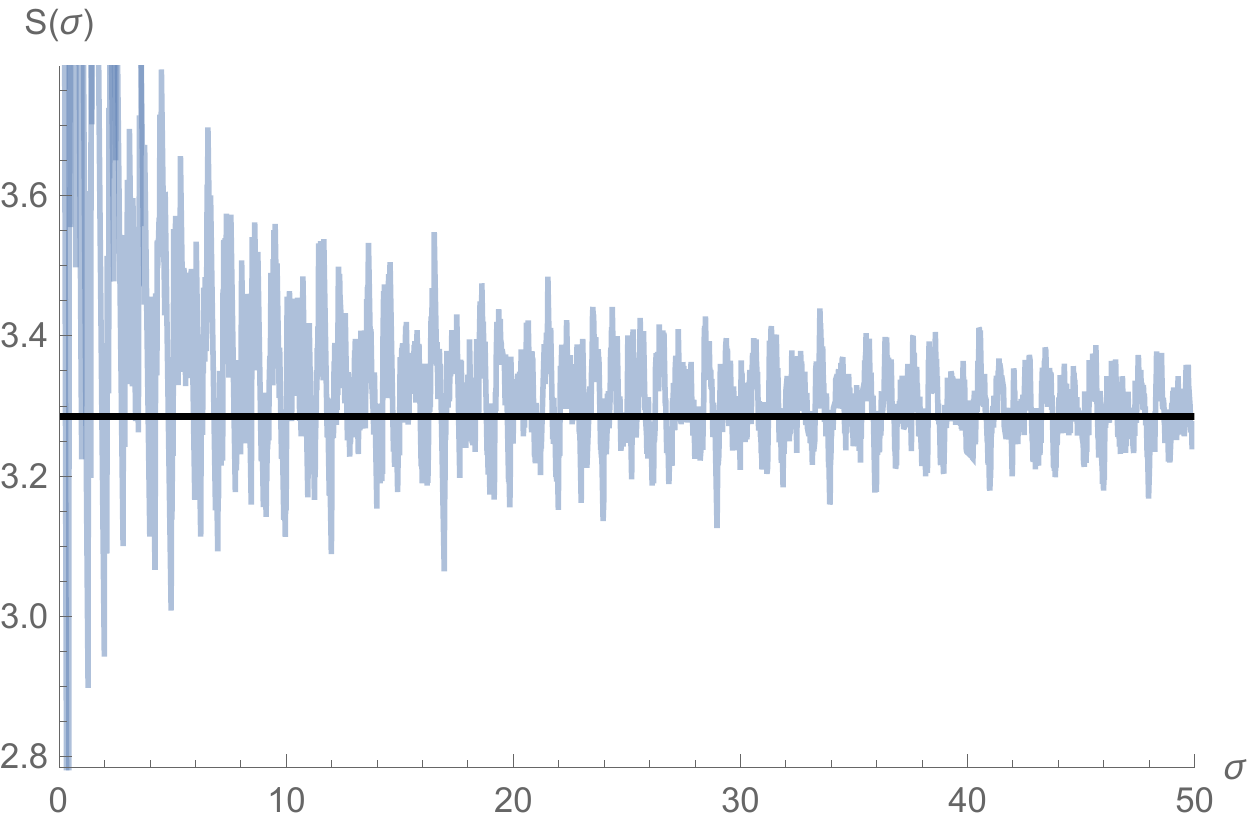}
\subcaption[second caption.]{$\alpha = \frac{\pi}{4}, \beta = \frac{\pi}{6}$}
\end{minipage} \vspace{0.5cm}

\captionsetup[subfigure]{labelformat=empty}
\centering
\begin{minipage}{0.5\textwidth}
  \centering
\includegraphics[width=0.9\textwidth]{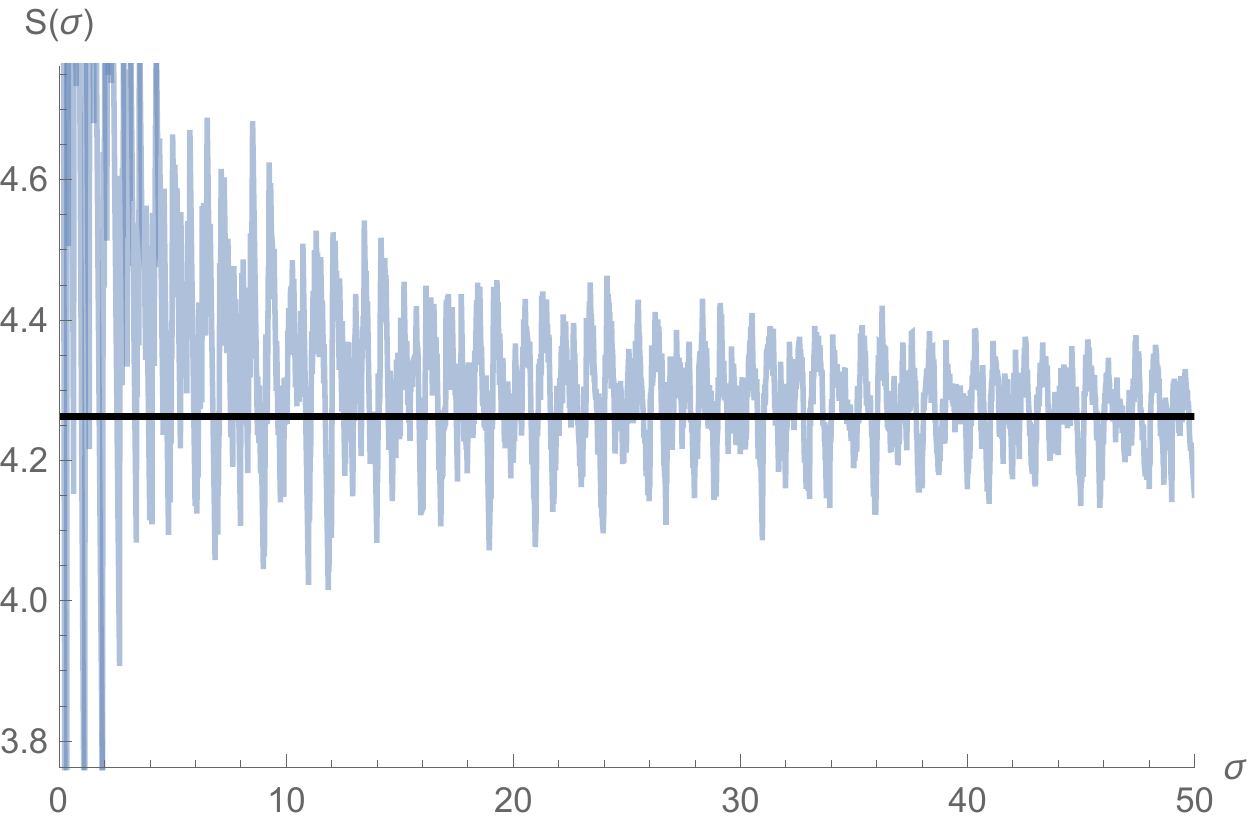}
\subcaption[first caption.]{$\alpha = \frac{\pi}{4}, \beta = \frac{\pi}{8}$}
\end{minipage}%
\begin{minipage}{0.5\textwidth}
  \centering
\includegraphics[width=0.9\textwidth]{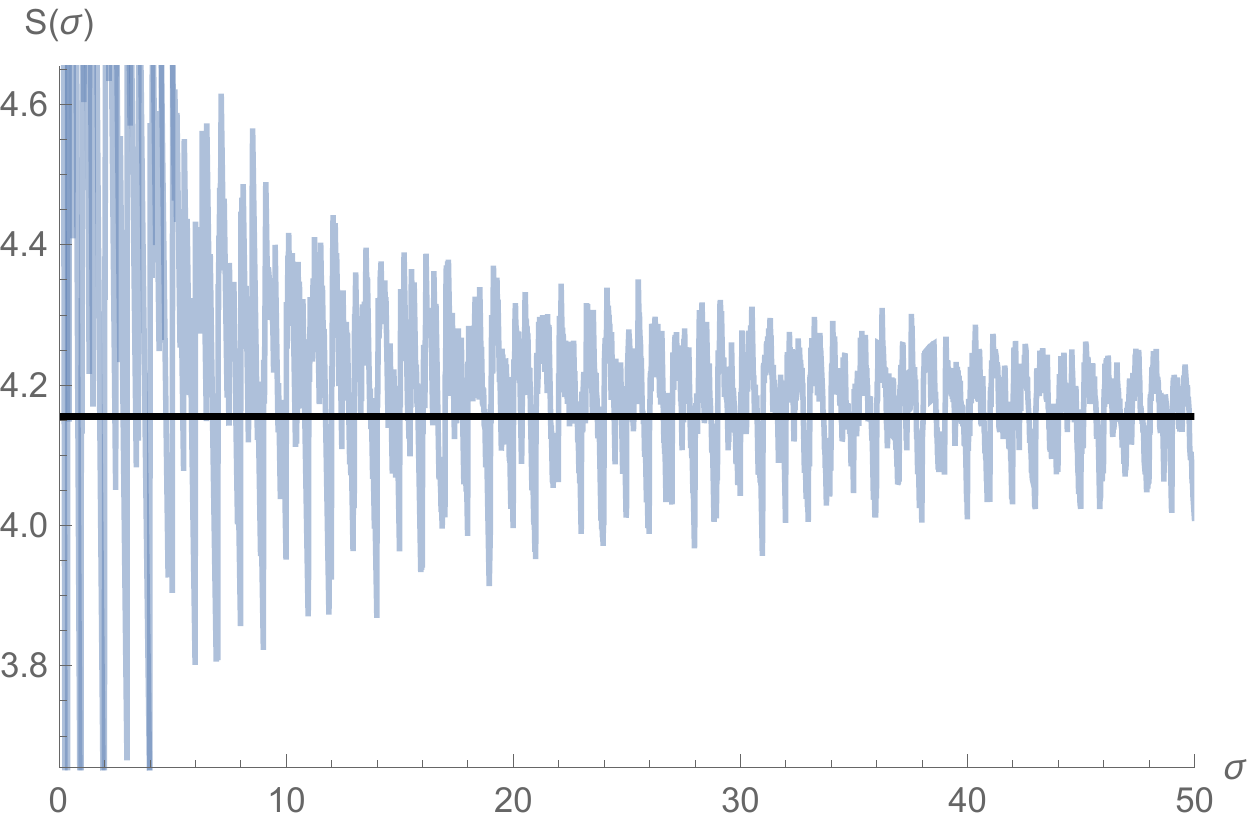}
\subcaption[second caption.]{$\alpha = \frac{\pi}{6}, \beta = \frac{\pi}{6}$}
\end{minipage} \vspace{0.5cm}

\captionsetup[subfigure]{labelformat=empty}
\centering
\begin{minipage}{0.5\textwidth}
  \centering
\includegraphics[width=0.9\textwidth]{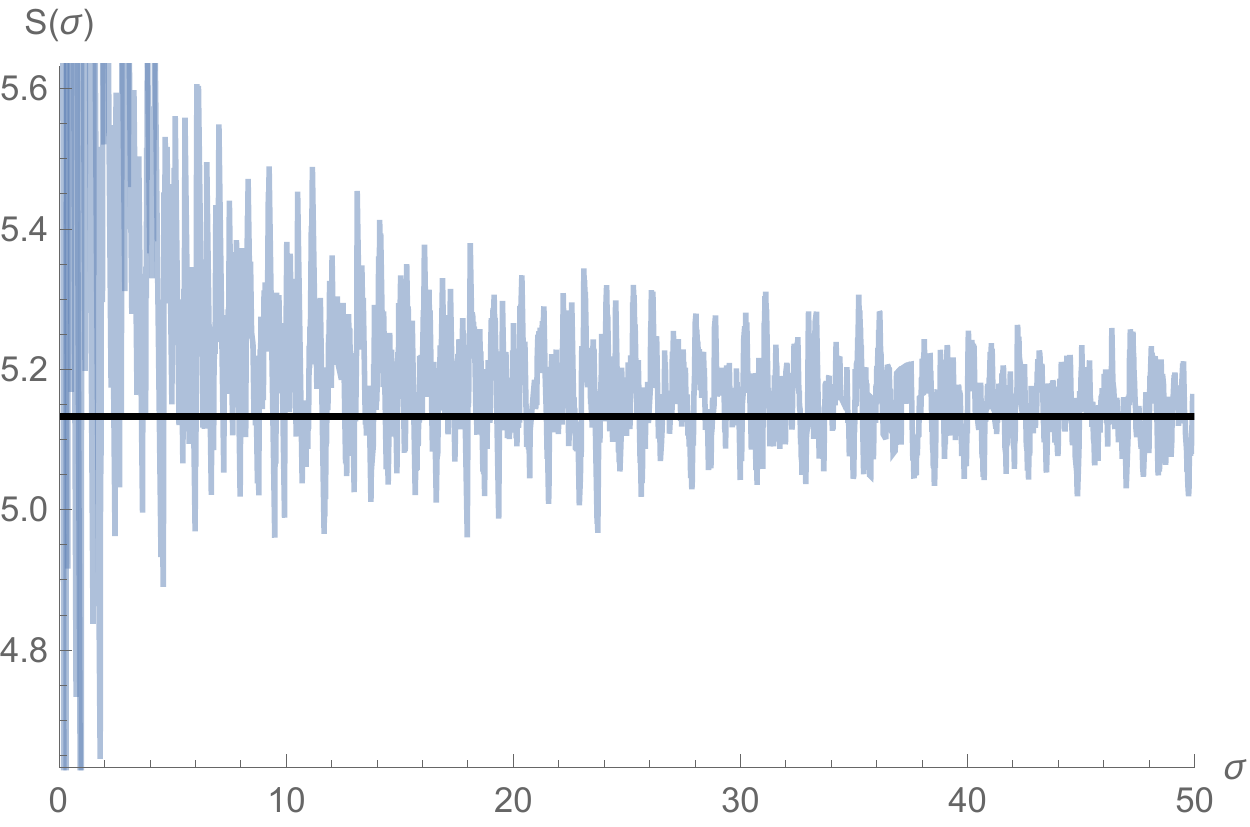}
\subcaption[first caption.]{$\alpha = \frac{\pi}{6}, \beta = \frac{\pi}{8}$}
\end{minipage}%
\begin{minipage}{0.5\textwidth}
  \centering
\includegraphics[width=0.9\textwidth]{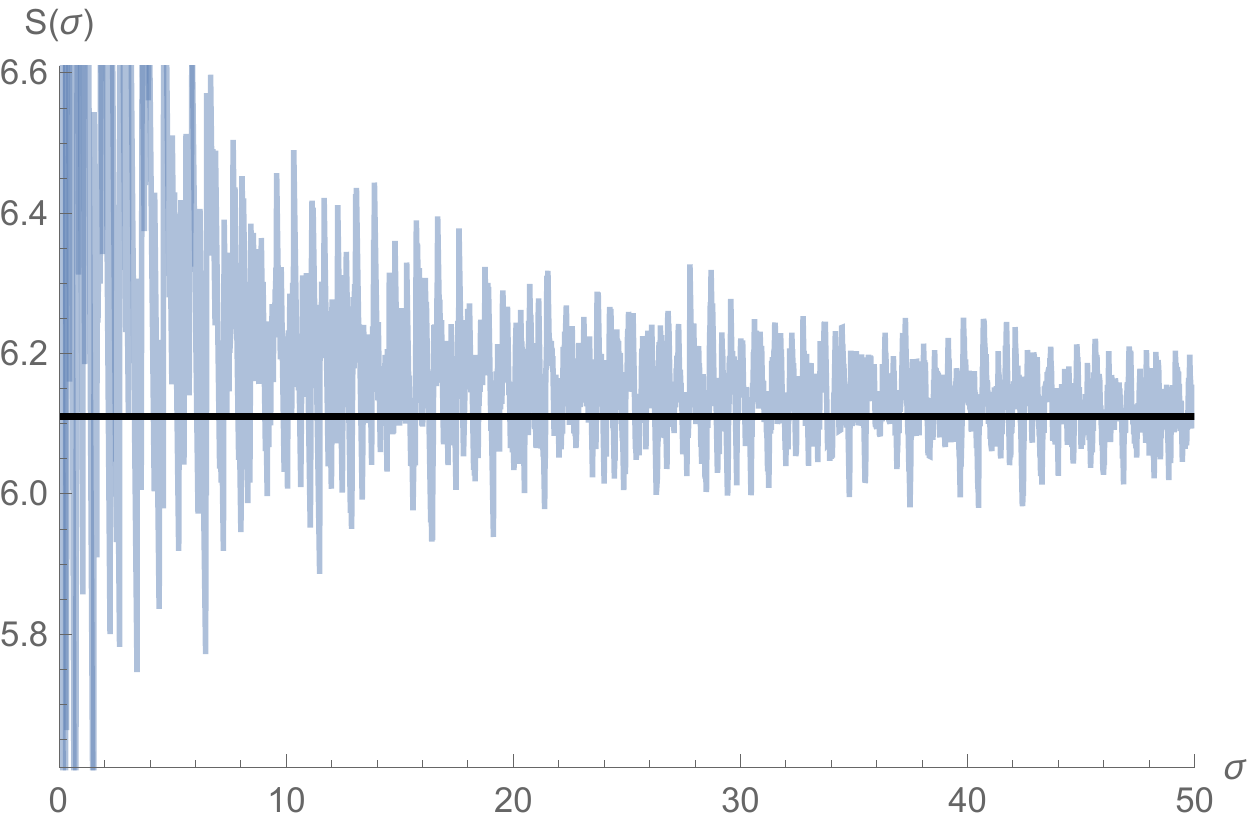}
\subcaption[second caption.]{$\alpha = \frac{\pi}{8}, \beta = \frac{\pi}{8}$}
\end{minipage} \vspace{0.5cm}
\caption{Value of $S(\sigma)$ compared to its conjectured limit indicated by the horizontal line.}\label{fig:Ssigma}
\end{figure}

\newpage

\sisetup{group-digits=false}

\sisetup{round-mode=places,table-figures-integer = 2, table-figures-decimal=4,round-precision=4,table-number-alignment=center,table-text-alignment=center}

\begin{center}
\begin{table}
\resizebox{0.95\columnwidth}{!}{
\begin{tabular}{c||c||c||c||c}
\toprule
\begin{minipage}{.22\linewidth}
\centering
\begin{tabular}{S[table-column-width = 1.4 cm] S[table-column-width = 1.4 cm]}
 0. & 0. \\
 0.25 & 0.437447 \\
 0.5 & 0.514336 \\
 0.707107 & 0.749535 \\
 1. & 0.998406 \\
 1.25 & 1.29113 \\
 1.41421 & 1.40574 \\
 1.5 & 1.49998 \\
 1.50486 & 1.52479 \\
 2. & 2. \\
 2.12132 & 2.12055 \\
 2.19003 & 2.19091 \\
 2.25 & 2.24321 \\
 2.42529 & 2.42043 \\
 2.5 & 2.5 \\
 2.82843 & 2.82838 \\
 2.90905 & 2.90722 \\
 3. & 3. \\
 3.09118 & 3.09086 \\
 3.25 & 3.25151 \\
 3.38136 & 3.38248 \\
 3.5 & 3.5 \\
 3.53553 & 3.53553 \\
 3.60779 & 3.60737 \\
 3.75468 & 3.75517 \\
\end{tabular}
\end{minipage} &

\begin{minipage}{.22\linewidth}
\centering
\begin{tabular}{S[table-column-width = 1.4 cm] S[table-column-width = 1.4 cm]}
 4. & 4. \\
 4.05596 & 4.0559 \\
 4.24264 & 4.24264 \\
 4.25 & 4.24999 \\
 4.32015 & 4.32028 \\
 4.3546 & 4.35459 \\
 4.43323 & 4.43315 \\
 4.5 & 4.5 \\
 4.65705 & 4.65705 \\
 4.94975 & 4.94975 \\
 5. & 5. \\
 5.02255 & 5.02258 \\
 5.0396 & 5.03959 \\
 5.1284 & 5.1284 \\
 5.25 & 5.25 \\
 5.32785 & 5.32784 \\
 5.33676 & 5.33675 \\
 5.5 & 5.5 \\
 5.59025 & 5.59025 \\
 5.65685 & 5.65685 \\
 5.73285 & 5.73285 \\
 5.81675 & 5.81675 \\
 5.99244 & 5.99244 \\
 6. & 6. \\
 6.03043 & 6.03043 \\
\end{tabular}
\end{minipage} &

\begin{minipage}{.22\linewidth}
\centering
\begin{tabular}{S[table-column-width = 1.4 cm]S[table-column-width = 1.4 cm]}
 6.25 & 6.25 \\
 6.26035 & 6.26034 \\
 6.32406 & 6.32406 \\
 6.36396 & 6.36396 \\
 6.43682 & 6.43682 \\
 6.5 & 6.5 \\
 6.51931 & 6.51931 \\
 6.54189 & 6.54189 \\
 6.67113 & 6.67113 \\
 6.89161 & 6.89161 \\
 7. & 7. \\
 7.02463 & 7.02463 \\
 7.07107 & 7.07107 \\
 7.14617 & 7.14617 \\
 7.21318 & 7.21318 \\
 7.2145 & 7.2145 \\
 7.25 & 7.25 \\
 7.31458 & 7.31458 \\
 7.35657 & 7.35657 \\
 7.5 & 7.5 \\
 7.50537 & 7.50537 \\
 7.56512 & 7.56512 \\
 7.77817 & 7.77817 \\
 7.81418 & 7.81418 \\
 7.85097 & 7.85097 \\
\end{tabular}
\end{minipage} &

\begin{minipage}{.22\linewidth}
\centering
\begin{tabular}{S[table-column-width = 1.4 cm]S[table-column-width = 1.4 cm]}
 7.91848 & 7.91848 \\
 8. & 8. \\
 8.02066 & 8.02066 \\
 8.04173 & 8.04173 \\
 8.18175 & 8.18175 \\
 8.22909 & 8.22909 \\
 8.25 & 8.25 \\
 8.30723 & 8.30724 \\
 8.47686 & 8.47687 \\
 8.48528 & 8.48528 \\
 8.48571 & 8.48571 \\
 8.5 & 8.5 \\
 8.5598 & 8.5598 \\
 8.61566 & 8.61566 \\
 8.73639 & 8.73639 \\
 8.75301 & 8.75301 \\
 8.90804 & 8.90804 \\
 9. & 9. \\
 9.01778 & 9.01778 \\
 9.12691 & 9.12691 \\
 9.15737 & 9.15737 \\
 9.19239 & 9.19239 \\
 9.25 & 9.25 \\
 9.2651 & 9.2651 \\
 9.30138 & 9.30139 \\
\end{tabular}
\end{minipage} &

\begin{minipage}{.22\linewidth}
\centering
\begin{tabular}{S[table-column-width = 1.4 cm]S[table-column-width = 1.4 cm]}
 9.32219 & 9.32219 \\
 9.4246 & 9.4246 \\
 9.42617 & 9.42617 \\
 9.45402 & 9.45402 \\
 9.5 & 9.5 \\
 9.58812 & 9.58812 \\
 9.70356 & 9.70356 \\
 9.80195 & 9.80195 \\
 9.89949 & 9.8995 \\
 9.97359 & 9.9736 \\
 10. & 10. \\
 10.0156 & 10.0156 \\
 10.0215 & 10.0215 \\
 10.0434 & 10.0434 \\
 10.1256 & 10.1256 \\
 10.1386 & 10.1386 \\
 10.25 & 10.25 \\
 10.2718 & 10.2718 \\
 10.2966 & 10.2966 \\
 10.3764 & 10.3764 \\
 10.4353 & 10.4353 \\
 10.4654 & 10.4654 \\
 10.5 & 10.5 \\
 10.6066 & 10.6066 \\
 10.6628 & 10.6628 \\
\end{tabular}
\end{minipage}
\\ \bottomrule
\end{tabular}
}
\caption{The first 125 quasi-eigenvalues (on the left) and sloshing eigenvalues (on the right) for $\alpha = \frac{\pi}{4}, \beta = \frac{\pi}{6}$.}\label{table:pi4pi6}
\end{table}
\end{center}

\begin{center}
\begin{table}
\resizebox{0.95\columnwidth}{!}{
\begin{tabular}{c||c||c||c||c}
\toprule
\begin{minipage}{.22\linewidth}
\centering
\begin{tabular}{S[table-column-width = 1.4 cm] S[table-column-width = 1.4 cm]}
 0. & \\
 0.309017 & 0. \\
 0.309017 & 0.221632 \\
 0.618034 & 0.282228 \\
 0.618034 & 0.383442 \\
 0.809017 & 0.601056 \\
 0.809017 & 0.629879 \\
 0.927051 & 0.705915 \\
 0.927051 & 0.872692 \\
 1. & 0.925498 \\
 1. & 0.928296 \\
 1.23607 & 1.23595 \\
 1.23607 & 1.23617 \\
 1.54508 & 1.39509 \\
 1.54508 & 1.46322 \\
 1.61803 & 1.52691 \\
 1.61803 & 1.54508 \\
 1.8541 & 1.54509 \\
 1.8541 & 1.8541 \\
 2. & 1.8541 \\
 2. & 1.91912 \\
 2.1275 & 2.16312 \\
 2.16312 & 2.16312 \\
 2.16312 & 2.16707 \\
 2.42705 & 2.29075 \\
\end{tabular}
\end{minipage} &

\begin{minipage}{.22\linewidth}
\centering
\begin{tabular}{S[table-column-width = 1.4 cm] S[table-column-width = 1.4 cm]}
 2.42705 & 2.3299 \\
 2.47214 & 2.47214 \\
 2.47214 & 2.47214 \\
 2.50772 & 2.52793 \\
 2.78115 & 2.66172 \\
 2.78115 & 2.78115 \\
 3. & 2.78115 \\
 3. & 3.08205 \\
 3.09017 & 3.09017 \\
 3.09017 & 3.09017 \\
 3.11168 & 3.18722 \\
 3.1222 & 3.20539 \\
 3.23607 & 3.26361 \\
 3.23607 & 3.27377 \\
 3.39919 & 3.39919 \\
 3.39919 & 3.39919 \\
 3.43965 & 3.49801 \\
 3.7082 & 3.7082 \\
 3.7082 & 3.7082 \\
 3.96168 & 3.99851 \\
 4. & 4.01722 \\
 4. & 4.01722 \\
 4.01722 & 4.02475 \\
 4.01722 & 4.03812 \\
 4.04508 & 4.05146 \\
\end{tabular}
\end{minipage} &

\begin{minipage}{.22\linewidth}
\centering
\begin{tabular}{S[table-column-width = 1.4 cm]S[table-column-width = 1.4 cm]}
 4.04508 & 4.09086 \\
 4.09388 & 4.11682 \\
 4.32624 & 4.32624 \\
 4.32624 & 4.32624 \\
 4.36953 & 4.38784 \\
 4.63525 & 4.63526 \\
 4.63525 & 4.63526 \\
 4.64126 & 4.65982 \\
 4.80983 & 4.82251 \\
 4.8541 & 4.85272 \\
 4.8541 & 4.85542 \\
 4.94427 & 4.94427 \\
 4.94427 & 4.94427 \\
 5. & 5.01247 \\
 5. & 5.0272 \\
 5.08004 & 5.09158 \\
 5.25329 & 5.25329 \\
 5.25329 & 5.25329 \\
 5.31585 & 5.32503 \\
 5.38942 & 5.39714 \\
 5.43566 & 5.44318 \\
 5.56231 & 5.56231 \\
 5.56231 & 5.56231 \\
 5.66312 & 5.66286 \\
 5.66312 & 5.66337 \\
\end{tabular}
\end{minipage} &

\begin{minipage}{.22\linewidth}
\centering
\begin{tabular}{S[table-column-width = 1.4 cm]S[table-column-width = 1.4 cm]}
 5.69526 & 5.70158 \\
 5.87132 & 5.87132 \\
 5.87132 & 5.87132 \\
 6. & 6.00315 \\
 6. & 6.00815 \\
 6.06947 & 6.07243 \\
 6.0795 & 6.08371 \\
 6.18034 & 6.18034 \\
 6.18034 & 6.18034 \\
 6.20039 & 6.20418 \\
 6.2748 & 6.27727 \\
 6.30633 & 6.3089 \\
 6.47214 & 6.47209 \\
 6.47214 & 6.47218 \\
 6.48936 & 6.48936 \\
 6.48936 & 6.48936 \\
 6.60734 & 6.60916 \\
 6.79837 & 6.79838 \\
 6.79837 & 6.79838 \\
 6.8106 & 6.81262 \\
 6.85278 & 6.85488 \\
 7. & 7.00177 \\
 7. & 7.00253 \\
 7.05416 & 7.05537 \\
 7.06124 & 7.06291 \\
\end{tabular}
\end{minipage} &

\begin{minipage}{.22\linewidth}
\centering
\begin{tabular}{S[table-column-width = 1.4 cm]S[table-column-width = 1.4 cm]}
 7.10739 & 7.1074 \\
 7.10739 & 7.1074 \\
 7.22465 & 7.22543 \\
 7.24277 & 7.24415 \\
 7.28115 & 7.28115 \\
 7.28115 & 7.28116 \\
 7.41641 & 7.41641 \\
 7.41641 & 7.41641 \\
 7.50537 & 7.50634 \\
 7.53825 & 7.53927 \\
 7.6 & 7.60073 \\
 7.68657 & 7.68755 \\
 7.72542 & 7.72543 \\
 7.72542 & 7.72543 \\
 7.93817 & 7.93885 \\
 8. & 8.00045 \\
 8. & 8.00081 \\
 8.03444 & 8.03445 \\
 8.03444 & 8.03445 \\
 8.05471 & 8.05514 \\
 8.09017 & 8.09017 \\
 8.09017 & 8.09017 \\
 8.17183 & 8.17205 \\
 8.2172 & 8.21756 \\
 8.22908 & 8.22948 \\
\end{tabular}
\end{minipage}
\\ \bottomrule
\end{tabular}
}
\caption{The first 125 quasi-eigenvalues (on the left) and sloshing eigenvalues (on the right) for $\alpha = \frac{\pi}{10}, \beta = \frac{\pi}{10}$.}\label{table:pi10pi10}
\end{table}
\end{center}

\begin{center}
\begin{table}
\resizebox{0.95\columnwidth}{!}{
\begin{tabular}{c||c||c||c||c}
\toprule
\begin{minipage}{.22\linewidth}
\centering
\begin{tabular}{S[table-column-width = 1.4 cm] S[table-column-width = 1.4 cm]}
 0. & \\
 0.173648 & \\
 0.347296 & \\
 0.5 & \\
 0.5 & 0. \\
 0.5 & 0.159981 \\
 0.520945 & 0.198515 \\
 0.694593 & 0.346859 \\
 0.766044 & 0.347096 \\
 0.868241 & 0.520943 \\
 0.939693 & 0.62304 \\
 1. & 0.694593 \\
 1. & 0.735768 \\
 1. & 0.868241 \\
 1.04189 & 0.879955 \\
 1.07607 & 1.04189 \\
 1.21554 & 1.05731 \\
 1.38919 & 1.18983 \\
 1.5 & 1.21554 \\
 1.5 & 1.30478 \\
 1.5 & 1.38919 \\
 1.53209 & 1.46934 \\
 1.55011 & 1.5081 \\
 1.56283 & 1.56283 \\
 1.73648 & 1.67199 \\
 1.87939 & 1.73648 \\
 1.91013 & 1.87183 \\
 2. & 1.91013 \\
 2. & 1.98432 \\
 2. & 1.99419 \\
 2.08378 & 2.00065 \\
 2.25743 & 2.08378 \\
 2.29813 & 2.25743 \\
 2.43107 & 2.30795 \\
 2.5 & 2.33165 \\
 2.5 & 2.43107 \\
 2.5 & 2.49883 \\
 2.60472 & 2.50003 \\
 2.6057 & 2.60472 \\
 2.77837 & 2.6917 \\
 2.81908 & 2.77837 \\
 2.92174 & 2.79317 \\
 2.95202 & 2.93438 \\
 3. & 2.95202 \\
 3. & 2.99975 \\
 3. & 3. \\
 3.06418 & 3.0638 \\
 3.12567 & 3.09746 \\
 3.29932 & 3.12567 \\
 3.44139 & 3.29932 \\
\end{tabular}
\end{minipage} &

\begin{minipage}{.22\linewidth}
\centering
\begin{tabular}{S[table-column-width = 1.4 cm] S[table-column-width = 1.4 cm]}
 3.47296 & 3.47297 \\
 3.5 & 3.49995 \\
 3.5 & 3.5 \\
 3.5 & 3.60044 \\
 3.59445 & 3.60287 \\
 3.64661 & 3.64661 \\
 3.75877 & 3.69197 \\
 3.82026 & 3.76193 \\
 3.83022 & 3.82026 \\
 3.87386 & 3.83016 \\
 3.99391 & 3.95674 \\
 4. & 3.99391 \\
 4. & 3.99999 \\
 4. & 4. \\
 4.14576 & 4.16756 \\
 4.16756 & 4.28527 \\
 4.32563 & 4.34121 \\
 4.3412 & 4.38849 \\
 4.5 & 4.5 \\
 4.5 & 4.5 \\
 4.5 & 4.51486 \\
 4.51485 & 4.54912 \\
 4.58171 & 4.59627 \\
 4.59627 & 4.62801 \\
 4.6885 & 4.68851 \\
 4.69846 & 4.69133 \\
 4.82312 & 4.86186 \\
 4.86215 & 4.86216 \\
 4.9277 & 4.97022 \\
 5. & 5. \\
 5. & 5. \\
 5. & 5.03581 \\
 5.0358 & 5.11634 \\
 5.20945 & 5.20946 \\
 5.21334 & 5.24209 \\
 5.36231 & 5.36231 \\
 5.38309 & 5.38311 \\
 5.5 & 5.5 \\
 5.5 & 5.50001 \\
 5.5 & 5.51799 \\
 5.55674 & 5.55676 \\
 5.57119 & 5.58823 \\
 5.63816 & 5.63575 \\
 5.65082 & 5.67664 \\
 5.73039 & 5.73042 \\
 5.73531 & 5.75428 \\
 5.78187 & 5.79637 \\
 5.90404 & 5.90407 \\
 6. & 6. \\
 6. & 6.00001 \\
\end{tabular}
\end{minipage} &

\begin{minipage}{.22\linewidth}
\centering
\begin{tabular}{S[table-column-width = 1.4 cm]S[table-column-width = 1.4 cm]}
 6. & 6.04265 \\
 6.07769 & 6.07772 \\
 6.1237 & 6.12836 \\
 6.12836 & 6.13481 \\
 6.25133 & 6.25138 \\
 6.36803 & 6.37924 \\
 6.42498 & 6.42503 \\
 6.46384 & 6.47807 \\
 6.5 & 6.5 \\
 6.5 & 6.50001 \\
 6.5 & 6.50682 \\
 6.56279 & 6.56931 \\
 6.57785 & 6.57736 \\
 6.58394 & 6.59164 \\
 6.59863 & 6.59869 \\
 6.749 & 6.75468 \\
 6.77228 & 6.77234 \\
 6.8944 & 6.89441 \\
 6.94593 & 6.946 \\
 7. & 7. \\
 7. & 7.00002 \\
 7. & 7.01522 \\
 7.05237 & 7.0569 \\
 7.0895 & 7.09552 \\
 7.11958 & 7.11966 \\
 7.14699 & 7.15187 \\
 7.29322 & 7.29331 \\
 7.33913 & 7.34634 \\
 7.46331 & 7.46662 \\
 7.46687 & 7.46697 \\
 7.5 & 7.5 \\
 7.5 & 7.50003 \\
 7.5 & 7.50381 \\
 7.51754 & 7.51747 \\
 7.55605 & 7.5597 \\
 7.64052 & 7.64063 \\
 7.66044 & 7.66046 \\
 7.72258 & 7.72578 \\
 7.7963 & 7.79916 \\
 7.81417 & 7.8143 \\
 7.8794 & 7.88239 \\
 7.96999 & 7.97224 \\
 7.98782 & 7.98796 \\
 7.99483 & 7.99741 \\
 8. & 8. \\
 8. & 8.00003 \\
 8. & 8.00553 \\
 8.16146 & 8.16162 \\
 8.25555 & 8.25899 \\
 8.33511 & 8.33529 \\
\end{tabular}
\end{minipage} &

\begin{minipage}{.22\linewidth}
\centering
\begin{tabular}{S[table-column-width = 1.4 cm]S[table-column-width = 1.4 cm]}
 8.36551 & 8.36743 \\ 
 8.42649 & 8.42652 \\
 8.45723 & 8.45723 \\
 8.5 & 8.5 \\
 8.5 & 8.50005 \\
 8.5 & 8.50201 \\
 8.50876 & 8.50896 \\
 8.51597 & 8.51753 \\
 8.55056 & 8.55249 \\
 8.55959 & 8.56101 \\
 8.68241 & 8.68263 \\
 8.701 & 8.7027 \\
 8.72065 & 8.72205 \\
 8.82549 & 8.82681 \\
 8.85606 & 8.8563 \\
 8.94767 & 8.94906 \\
 9. & 9. \\
 9. & 9.00006 \\
 9. & 9.00207 \\
 9.02971 & 9.02997 \\
 9.19253 & 9.19257 \\
 9.19857 & 9.20013 \\
 9.20335 & 9.20364 \\
 9.21938 & 9.22024 \\
 9.28493 & 9.28597 \\
 9.29183 & 9.29265 \\
 9.36464 & 9.3655 \\
 9.377 & 9.37732 \\
 9.39693 & 9.39695 \\
 9.5 & 9.5 \\
 9.5 & 9.50008 \\
 9.5 & 9.50079 \\
 9.54602 & 9.54678 \\
 9.55065 & 9.551 \\
 9.59983 & 9.60045 \\
 9.68311 & 9.68379 \\
 9.7056 & 9.70634 \\
 9.7243 & 9.72468 \\
 9.89795 & 9.89836 \\
 9.90844 & 9.909 \\
 9.93757 & 9.93806 \\
 9.95858 & 9.95864 \\
 9.97268 & 9.97321 \\
 10. & 10. \\
 10. & 10.0001 \\
 10. & 10.0008 \\
 10.0716 & 10.0721 \\
 10.112 & 10.1124 \\
 10.1588 & 10.1595 \\
 10.2016 & 10.2021 \\
\end{tabular}
\end{minipage} &

\begin{minipage}{.22\linewidth}
\centering
\begin{tabular}{S[table-column-width = 1.4 cm]S[table-column-width = 1.4 cm]}
 10.2175 & 10.218 \\
 10.2452 & 10.2457 \\
 10.3366 & 10.3367 \\
 10.4189 & 10.4194 \\
 10.5 & 10.5 \\
 10.5 & 10.5001 \\
 10.5 & 10.5003 \\
 10.5069 & 10.5071 \\
 10.5422 & 10.5425 \\
 10.5925 & 10.5931 \\
 10.6047 & 10.6051 \\
 10.6397 & 10.64 \\
 10.6681 & 10.6683 \\
 10.7037 & 10.704 \\
 10.7246 & 10.7247 \\
 10.7645 & 10.7648 \\
 10.7662 & 10.7668 \\
 10.8753 & 10.8756 \\
 10.9398 & 10.9405 \\
 10.9668 & 10.967 \\
 11. & 11. \\
 11. & 11.0002 \\
 11. & 11.0003 \\
 11.0635 & 11.0637 \\
 11.1135 & 11.1142 \\
 11.1302 & 11.1305 \\
 11.1605 & 11.1607 \\
 11.2763 & 11.2764 \\
 11.2871 & 11.2879 \\
 11.3567 & 11.3569 \\
 11.3861 & 11.3863 \\
 11.4344 & 11.4345 \\
 11.4608 & 11.4616 \\
 11.4907 & 11.4908 \\
 11.5 & 11.5 \\
 11.5 & 11.5002 \\
 11.5 & 11.5002 \\
 11.509 & 11.5092 \\
 11.5189 & 11.519 \\
 11.539 & 11.5391 \\
 11.5868 & 11.587 \\
 11.6344 & 11.6354 \\
 11.6553 & 11.6554 \\
 11.8081 & 11.8091 \\
 11.8471 & 11.8472 \\
 11.8487 & 11.8488 \\
 11.9452 & 11.9453 \\
 11.9817 & 11.9828 \\
 12. & 12. \\
 12. & 12.0001 \\
\end{tabular}
\end{minipage}
\\ \bottomrule
\end{tabular}
}
\caption{The first 250 quasi-eigenvalues (on the left) and sloshing eigenvalues (on the right) for $\alpha = \frac{\pi}{6}, \beta = \frac{\pi}{18}$.}\label{table:pi6pi18}
\end{table}
\end{center}

\newpage

\nocite{}
\bibliographystyle{plain}
\bibliography{ref}
\end{document}